\documentclass[final,leqno]{siamltex704}
\usepackage{amsmath}
\usepackage{graphicx}
\usepackage[notcite,notref]{showkeys}
\usepackage{mathrsfs}
\usepackage{float}
\usepackage{amsfonts,amssymb}
\usepackage{dsfont}
\usepackage{pifont}
\usepackage{wrapfig} % Allows in-line images
\usepackage{multirow}
\numberwithin{equation}{section}
\def\3bar{{|\hspace{-.02in}|\hspace{-.02in}|}}
\def\E{{\mathcal{E}}}
\def\T{{\mathcal{T}}}

\def\w{\psi}

\def\bn{{\mathbf{n}}}

\newtheorem{defi}{Definition}[section]

\newtheorem{algorithm}{Weak Galerkin Algorithm}
\title {A Weak Galerkin Finite Element Method  for A Type of Fourth Order Problem Arising From  Fluorescence Tomography}

\author{
Chunmei Wang\thanks{School of Mathematics, Georgia Institute of
Technology, Atlanta, Georgia, 30332, USA; Taizhou College, Nanjing Normal
University, Taizhou, China 225300. The research of Chunmei Wang was
supported by National Science Foundation Award
DMS-1522586 and by Jiangsu Provincial Foundation Award
BK20050538.}   \and Haomin Zhou\thanks{School of Mathematics, Georgia Institute of Technology, Atlanta, GA 30332, USA. The research of Haomin Zhou  was  supported  by  NSF  Faculty  Early  Career  
Development(CAREER)  Award  DMS-0645266,  DMS-1042998,  DMS-1419027,  and ONR Award N000141310408.}}

\begin{document}

\maketitle

\begin{abstract}
In this paper, a new and efficient numerical algorithm by using weak Galerkin (WG) finite element
methods is proposed for a type of fourth order problem arising from  fluorescence tomography(FT).  Fluorescence tomography   is an emerging, in vivo non-invasive 3-D
 imaging technique which reconstructs images that characterize the
 distribution of molecules that are tagged by fluorophores. 
 Weak second order elliptic operator and its
 discrete version
 are introduced for a class
of discontinuous functions defined on a finite element partition of
the domain consisting of general polygons or polyhedra.  An error estimate
of optimal order is derived in an $H^2$-equivalent norm for the WG
finite element solutions. Error estimates in the usual $L^2$ norm
are established, yielding optimal order of convergence for all the
WG finite element algorithms except the one corresponding to the
lowest order (i.e., piecewise quadratic elements). Some numerical
experiments are presented to illustrate the efficiency and accuracy of the numerical scheme.

% Some numerical
% experiments are presented to illustrate the efficiency and accuracy
% of the numerical scheme.
\end{abstract}



\pagestyle{myheadings}

\section{Introduction}
This paper is concerned with new developments of numerical
methods for a type of fourth order  problem with Dirichlet and Neumann boundary conditions. The model problem seeks an unknown function
 $u=u(x)$ satisfying
\begin{equation}\label{0.1}
\begin{split}
( -\nabla\cdot(\kappa \nabla)+\mu)^2u&=f, \quad\text{in}\ \Omega,\\
u&=\xi, \quad\text{on}\ \partial\Omega,\\
\kappa  \nabla u\cdot \textbf{n} &=\nu, \quad  \text{on}\ \partial\Omega,\\
\end{split}
\end{equation}
where   $\Omega$ is an open bounded domain in $\mathbb{R}^d$($d=2,3$)
with a Lipschitz continuous boundary $\partial\Omega$, $\bn$ is the unit outward normal direction to $\partial\Omega$, $\kappa$ is a  
 symmetric and positive definite matrix-valued function, $\mu$ is  nonnegative real-valued function,
and the functions
$f$, $\xi$, and $\nu$ are given in the domain or on its boundary, as
appropriate. For convenience,   denote the second order elliptic operator  $ \nabla\cdot(\kappa \nabla)$ as $E$. For simplicity and without loss of generality,   
we assume that $\kappa$ is   piecewise constant matrix and $\mu$ is a piecewise constant.

The fourth order model problem (\ref{0.1}) arises from  fluorescence tomograph(FT)\cite{ yinkepaper, gkr2014, gz2013, ltka2015, mtspa2011, wbwm2003, wbwm2004,  yinkethesis}, which is an emerging, in vivo noninvasive 3-D
imaging technique.
FT captures molecular specific information by using highly specific fluorescent
probes and non-ionizing  NIR  radiation instead of X-ray or other powerful magnetic field \cite{p-25},  which makes FT a potentially less harmful  medical imaging modality
 compared to other medical imaging modalities, such as CT and MRI.  FT aims to 
reconstruct the
distribution of fluorophores, which are tagged with targetted molecules, from boundary measurements. Therefore,
FT has been regarded as a promising method in early cancer detection and drug monitoring nowadays
\cite{p-5, p-24, p-39}.   

We introduce the following space 
$$
H_{\kappa}^2(\Omega)=\{v:v\in H^1(\Omega),\kappa \nabla v\in H(\text{div};\Omega)\},
$$
which is equipped with the following norm
$$
\|v\|_{\kappa,2}=(\|v\|_1^2+\|\nabla\cdot (\kappa \nabla v)\|^2)^{\frac{1}{2}}.
$$

A variational formulation for the fourth order model problem (\ref{0.1}) is
given by seeking $u\in H_{\kappa}^2(\Omega)$ satisfying $u|_{\partial
\Omega}=\xi$, $\kappa  \nabla u\cdot \textbf{n}|_{\partial
\Omega}  =\nu$, such that
\begin{equation}\label{0.2}
 (Eu,Ev)+2\mu (\kappa \nabla u,\nabla v)+\mu^2(u,v)=(f,v), \quad
\forall v\in {\cal V},
\end{equation}
where  $(\cdot,\cdot)$ stands for the usual inner product in
$L^2(\Omega)$, and the test space ${\cal V}$  can be defined as
$$
{\cal V}=\{v\in H_\kappa^2(\Omega):v|_{\partial\Omega}=0,\kappa \nabla v\cdot \textbf{n}|_{\partial \Omega}=0\}.
$$
Here, $\bn$ is the unit outward normal direction to the boundary of $T$.

There have been various
conforming finite element schemes proposed for a general fourth order elliptic problem, such as the biharmonic equation,  by constructing
finite element spaces as subspaces of $H^2(\Omega)$. Such
$H^2$-conforming methods essentially require $C^1$-continuity for
the underlying piecewise polynomials (known as finite element
functions) on a prescribed finite element partition \cite{bs}.  The
$C^1$-continuity imposes an enormous difficulty in the construction
of the corresponding finite element functions in practical
computations due to the fact that the  $C^1$-continuous elements have high degrees of freedom. For example, the Argyris element has 21  degrees of freedom, 
and the Bell element has  18   degrees of freedom.
  Because of the complexity in the construction of
$C^1$-continuous elements, $H^2$-conforming finite element methods
are rarely used in practice for solving the biharmonic equation.
As an alternative approach, nonconforming and discontinuous Galerkin
finite element methods have been developed for solving the
biharmonic equation over the last several decades. The Morley
element \cite{m1968} is a well-known example of nonconforming
element for the biharmonic equation by using piecewise quadratic
polynomials. Recently, a $C^0$ interior penalty method was studied
in \cite{bs2005, eghlmt2002}. In \cite{mb2007}, a hp-version
interior-penalty discontinuous Galerkin method was developed for the
biharmonic equation. To avoid the use of $C^1$-elements, mixed
methods have been developed for the biharmonic equation by reducing
the fourth order problem to a system of two second order equations
\cite{ab1985, f1978, gnp2008, m1987, mwy3818}.

 The difference between (\ref{0.2}) and the standard
bi-harmonic equation is significant. First, the usual $H^2$
conforming elements designed for the bi-harmonic equation are no
longer $H_\kappa^2$-conforming, and thus are not applicable to the
  problem (\ref{0.2}). For some well-known non-conforming
finite elements, such as the Morley element \cite{m1968}, for the biharmonic
equation, the corresponding variational formulation involves the
full Hessian. Since it is not clear if the problem (\ref{0.2}) can
be re-formulated in an equivalent form, the applicability of such
non-conforming finite elements is highly questionable. In fact, 
we believe that they can not be directly applicable to the 
problem (\ref{0.2}).
The problem (\ref{0.2}) can also be formulated in a mixed form by
using an auxiliary variable $w=-\nabla\cdot ( \kappa \nabla u)+\mu u$. The exact mixed formulation
seeks $u, w \in H^1(\Omega)$ such that $u|_{\partial\Omega} = \xi$
and satisfying
\begin{equation}\label{BI-E-Mixed}
\begin{split}
(w, \phi) - (\kappa\nabla u, \nabla \phi) - (\mu u, \phi) & =
-\langle \nu, \phi\rangle_{\partial\Omega},\qquad \forall \phi\in
H^1(\Omega),\\
(\kappa\nabla w, \nabla v) + (\mu w, v) & = (f, v),\qquad \qquad \forall
v\in H_0^1(\Omega).
\end{split}
\end{equation}
Most of the existing finite element methods are applicable to the
mixed formulation (\ref{BI-E-Mixed}). One drawback with the mixed
formulation is the saddle-point nature of the problem, which causes
extra difficulty in the design of fast solution techniques for the
corresponding discretizations.

Recently, weak Galerkin (WG) finite element method is a new and efficient numerical
method for solving partial differential equations. The method/idea
was  first proposed by Junping Wang and Xiu Ye in 2011 for solving
second order elliptic problem in  \cite{wy2013}, and the concept
was further developed in \cite{mwy3655, ww2013, ww2014, ww2015, ww2016, wy3655, wy2707}. The central idea
of WG is to interpret partial differential operators as generalized
distributions, called weak differential operators, over the space of
discontinuous functions including boundary information. The weak
differential operators are further discretized and applied to the
corresponding variational formulations of the underlying PDEs. By
design, WG uses generalized and/or discontinuous approximating
functions on general meshes to overcome the barrier in the
construction of ``smooth'' finite element functions. The current research indicates that the concept of discrete
weak differential operators offers a new paradigm in numerical methods for partial
differential equations.

The proposed WG finite element algorithm for the fourth order problem (\ref{0.1}) is based on two new
ideas: (1) the computation of a discrete weak second order elliptic
operator locally on each element that takes
into account  the coefficient matrix from applications; and (2) a stabilizer that takes
into account the jump of the coefficient matrix from applications.
 The research is
innovative in that the proposed algorithm is the first ever finite
element method that works for the fourth order  problem (\ref{0.1}) based on the variational formulation (\ref{0.2}). 

There is a point which is of great significance   to point out in the process of our research. Actually, there is a  more straightforward 
 variational formulation for the fourth order problem  (\ref{0.1}) which is  easier to propose:  seeking $u\in H_{\kappa}^2(\Omega)$ satisfying $u|_{\partial
\Omega}=\xi$, $\kappa  \nabla u\cdot \textbf{n}|_{\partial
\Omega}  =\nu$, such that
\begin{equation}\label{0.2-2}
 (Fu,Fv)=(f,v), \quad
\forall v\in {\cal V},
\end{equation}
where $F= -\nabla\cdot(\kappa \nabla)+\mu$. However,  according to our research, we find  that variational formulation (\ref{0.2-2}) is  not an appropriate variational formulation where weak Galerkin finite element method
can  be applied   successfullly  and the code based on variational formulation (\ref{0.2-2}) cannot work well especially  when $\mu$ is not small enough. There is a straightforward reason from the formulation that 
there is no second order elliptic term in the variational formulation  (\ref{0.2-2}) which plays an important role  in the algorithm design. This fact leads to the result that
 the analysis of WG method cannot pass through based on the variational formulation  (\ref{0.2-2}) . 
 
The paper is organized as follows.  Section \ref{Section:Wpartial} is devoted to a
discussion of weak second order elliptic  operator, and weak gradient as well as  their discrete versions. In
Section \ref{Section:WGFEM}, we present a weak Galerkin algorithm
for the fourth order   model problem (\ref{0.1}) based on the variational formulation (\ref{0.2}). In Section
\ref{Section:L2projection}, we introduce some local $L^2$ projection
operators and then derive some approximation properties which are
useful in the convergence analysis. Section
\ref{Section:error-equation} will be devoted to the derivation of an
error equation for the WG finite element solution. In Section
\ref{Section:H2error}, we shall  establish an optimal order of error
estimate for the WG finite element approximation in a
$H^2$-equivalent discrete norm. In Section \ref{Section:L2error}, we
shall derive error estimates for the WG finite element  solution
 in the usual $L^2$-norm. 
Finally
in Section 8, we present some numerical results that verify the theory established in
the previous sections.

Throughout the paper, the letter $C$ is used to denote a generic
constant independent of the mesh size and functions involved.

\section{Weak Differential Operators  and Discrete Weak Differential Operators}\label{Section:Wpartial} For the fourth order    model problem
(\ref{0.1}) with the variational form (\ref{0.2}), the principle
differential operators are the   second order elliptic  operator  $E$,  and the gradient operator. Thus, we shall define
weak second order elliptic operator   and  review  the definition of weak  gradient operator,  which was first introduced in \cite{wy3655}, for a
class of discontinuous functions. For numerical purpose, we shall
also introduce a discrete version for the weak  second order elliptic  operator 
  and  review the definition of discrete  weak  gradient operator  \cite{wy3655} in the polynomial subspaces.

Let ${\cal T}_h$ be a partition of the domain $\Omega$ into polygons
in 2D or polyhedra in 3D. Assume that ${\cal T}_h$ is shape regular
in the sense  as defined in \cite{wy3655}, which is more complex than the standard definition of shape regularity in the theory of finite element methods. Denote by $\E_h$ the set
of all edges or flat faces in ${\cal T}_h$, and let
$\E_h^0=\E_h\setminus\partial\Omega$ be the set of all interior
edges or flat faces in ${\cal T}_h$.  
By a weak function on the region $T$, we mean a
function $v=\{v_0,v_b,v_g\}$ such that $v_0\in L^2(T)$,
$v_b\in L^{2}(\partial T)$ and $v_g\in L^{2}(\partial
T)$. The first and second components $v_0$ and $v_b$ can be
understood as the value of $v$ in the interior and on the boundary
of $T$. The third term $v_g $  intends to represent  the value of
$\kappa\nabla v \cdot \textbf{n}$ on the boundary of $T$, where $\textbf{n}$ is the unit outward normal direction to the boundary of $T$.
  On each  interior edge or flat face 
$e\in {\cal E}_h^0$ shared by two elements denoted by $T_L$ and $T_R$, $v_g$ has two copies of value: one as seen from the left-hand side element $T_L$ 
denoted by $v_g^L$, and the other as seen
 from the right-hand side element  $T_R$  denoted by $v_g^R$, such that $v_g^L + v_g^R = 0$.

Denote by $W(T)$ the space of all weak functions on $T$; i.e.,
$$
W(T)=\{v=\{v_0,v_b,v_g\}: v_0\in L^2(T), v_b\in
L^{2}(\partial T), v_g\in L^{2}(\partial T) \}.
$$ 

\begin{defi}\label{defition3.1}   The dual of $L^2(T)$ can be identified with
itself by using the standard $L^2$ inner product as the action of
linear functionals. With a similar interpretation,  for any $v\in
W(T)$, the weak second order elliptic  operator   of
$v=\{v_0,v_b,v_g \}$,   denoted  by
$E_{ w} v$, is defined as a linear functional in the dual space of $H^2(T)$ whose action on
each $\varphi \in H^2(T)$ is given by
 \begin{equation}\label{2.3}
 (E_{ w}v,\varphi)_T=(v_0,E\varphi)_T-
 \langle v_b ,\kappa \nabla \varphi\cdot  \textbf{n}\rangle_{\partial T}+
 \langle v_{g },\varphi  \rangle_{\partial T}.
 \end{equation}
Here,  $\langle\cdot,\cdot\rangle_{\partial T}$ stands for the usual inner product in
$L^2(\partial T)$.
\end{defi}

For computational purpose, we introduce a discrete version of the weak  second order  elliptic operator 
  by approximating $E_{ w}$ in a 
polynomial subspace of the dual of $H^2(T)$. To this end, for any
non-negative integer $r\geq 0$, denote by $P_r(T)$ the set of
polynomials on $T$ with degree no more than $r$. A discrete
 weak second order elliptic  operator, denoted by
$E_{ w,r,T}$, is defined as the unique polynomial
 $E_{ w,r,T} v\in P_r(T)$ satisfying the following equation
  \begin{equation}\label{2.4}
  (E_{ w}v,\varphi)_T=(v_0,E\varphi)_T-
 \langle v_b ,\kappa \nabla \varphi\cdot  \textbf{n}\rangle_{\partial T}+
 \langle v_{g },\varphi  \rangle_{\partial T},\quad \forall \varphi \in
 P_r(T), 
 \end{equation}
which follows from  the usual integration by parts that
 \begin{equation}\label{A.002}
(E_{w}v,
\varphi)_T=(Ev_0,\varphi)_T+\langle
v_0-v_b, \kappa\nabla\varphi\cdot \textbf{n} \rangle_{\partial T}-\langle
 \kappa \nabla v_0\cdot \textbf{n}-v_{g },\varphi  \rangle_{\partial T},\forall \varphi \in
 P_r(T).
\end{equation}

\begin{defi} \cite{wy3655}  The dual of $L^2(T)$ can be identified with
itself by using the standard $L^2$ inner product as the action of
linear functionals. With a similar interpretation,  the weak gradient  operator  of any  $v\in
W(T)$, denoted  by $ \nabla_{ w} v$,  is defined as a linear  vector functional  in the dual space of $[H^1(T)]^d$ whose action on
each $\boldsymbol{ \psi} \in [H^1(T)]^d$ is given by
 \begin{equation}\label{2.3-2}
 ( \nabla_{ w}v,\boldsymbol{ \psi})_T=-(v_0, \nabla \cdot \boldsymbol{ \psi})_T+
 \langle v_b ,\boldsymbol{ \psi}\cdot  \textbf{n}\rangle_{\partial T}.
 \end{equation}
\end{defi}

 A discrete
 weak gradient  operator, denoted by
$ \nabla_{ w,r,T}$, is defined as the unique vector polynomial
 $ \nabla_{ w,r,T} v\in [P_r(T)]^d$ satisfying the following equation
  \begin{equation}\label{2.4-2}
  ( \nabla_{ w}v,\boldsymbol{ \psi})_T=-(v_0, \nabla \cdot \boldsymbol{ \psi})_T+
 \langle v_b ,\boldsymbol{ \psi}\cdot  \textbf{n}\rangle_{\partial T},\quad \forall \boldsymbol{ \psi}\in
 [P_r(T)]^d, 
 \end{equation}
which  follows from the usual integration by parts that
 \begin{equation}\label{2.4-3}
  ( \nabla_{ w}v,\boldsymbol{ \psi})_T= (\nabla v_0,  \boldsymbol{ \psi})_T-
 \langle v_0-v_b ,\boldsymbol{ \psi}\cdot  \textbf{n}\rangle_{\partial T},\quad \forall \boldsymbol{ \psi}\in
 [P_r(T)]^d. 
 \end{equation}
\section{ Numerical Algorithm by Weak Galerkin}\label{Section:WGFEM}

For any given integer $k\geq 2$, denote by $W_k(T)$ the local  discrete
weak function space given by
\begin{equation*}
W_k(T)=\big\{v=\{v_0,v_b, v_g\}: v_0\in P_k(T), v_b\in
P_{k }(e),v_g\in  P_{k-1}(e), e\subset \partial
T\big\}.
\end{equation*}

Patching $W_k(T)$ over all the elements $T\in {\cal T}_h$ through
 the interface $\E_h^0$, we arrive at a weak finite
element space $V_h$ defined as follows
$$
V_h=\big\{v=\{v_0,v_b,v_g\}:\{v_0,v_b,v_g\}|_T\in
W_k(T), \forall T\in {\cal T}_h\big\}.
$$

Denote by $V_h^0$ the subspace of $V_h$ with vanishing trace; i.e.,
$$
V_h^0=\{v=\{v_0,v_b,v_g\}\in
V_h,v_b|_e=0,v_g|_e=0, e\subset \partial T\cap
\partial\Omega\}.
$$

Denote by
$E_{  w,k-2}$ and  $ \nabla_{  w,k-1}$ the discrete weak second order elliptic  operator  and  the discrete  weak gradient  operator  on the finite element
space $V_h$
computed by using (\ref{2.4}) and (\ref{2.4-2}) on each element $T$ for $k\geq 2$, respectively;
i.e.,
$$
(E_{ w,k-2} v)|_T=E_{ w,k-2,T}(v|_T), \qquad
v\in V_h,
$$
 $$
( \nabla_{ w,k-1} v)|_T= \nabla_{ w,k-1,T}(v|_T), \qquad
v\in V_h.
$$

For simplicity of notation and without confusion, we shall drop the
subscripts $k-2$ and $k-1$ in the notations $E_{ w,k-2}$ and $ \nabla_{ w,k-1} $, respectively. For any $u=\{u_0,u_b, u_g\}$ and
$v=\{v_0,v_b, v_g\}$ in $V_h$, we  
introduce the following notations and  a bilinear form as follows
$$
(E_{w}u,E_{w}v)_h=\sum_{T\in{\cal
T}_h} (E_{ w}u,E_{ w}v)_T, 
$$
$$
  (\kappa \nabla_w u,\nabla_w v)_h=\sum_{T\in{\cal
T}_h}(\kappa \nabla_w u,\nabla_w v)_T, 
 $$ 
$$
s(u,v)= \sum_{T\in {\cal T}_h}  h_T^{-1}\langle  \kappa \nabla
u_0 \cdot \textbf{n}- u_g, \kappa \nabla
v_0 \cdot \textbf{n}-
v_g\rangle_{\partial T}+\sum_{T\in {\cal T}_h} h_T^{-3}\langle  u_0-u_b,  
v_0-v_b\rangle_{\partial T}.
$$

For each element $T$, denote by $Q_0$ the  $L^2$ projection  onto $P_{k}(T)$. For each edge or face $e\subset\partial T$,
denote by $Q_{b}$ and $Q_{g}$ the $L^2$ projections onto $P_{k}(e)$ and $P_{k-1}(e)$, respectively. Now, for any $u\in H^2(\Omega)$, 
we can define a projection onto the weak finite element space $V_h$
such that on each element $T$,
$$
Q_hu=\{Q_0u,Q_bu,Q_g(\kappa \nabla u \cdot \textbf{n})\}.
$$ 

The following is a precise statement of the WG finite element scheme
for the fourth order  model problem (\ref{0.1}) based on the variational
formulation (\ref{0.2}).

\begin{algorithm} Find $u_h=\{u_0,u_b, u_g\}\in V_h$
satisfying $u_b=Q_{b}\xi$ and $u_g =Q_{g}\nu$  on
$\partial\Omega$, such that
\begin{equation}\label{2.7}
(E_{w} u_h,E_{w}v)_h+2\mu (\kappa \nabla_w u_h,\nabla_w v)_h+\mu^2(u_h,v) +s(u_h,v)=(f,v_0),  
\forall v \in V_h^0.
\end{equation}
\end{algorithm}

The following is a useful observation concerning the finite element
space $V_h^0$.

\begin{lemma}\label{Lemma4.1} For any $v\in V_h^0$, define $\3barv\3bar$ by
\begin{equation}\label{3barnorm}
\3barv\3bar^2=(E_{w} v,E_{w}v)_h+2\mu (\kappa \nabla_w v,\nabla_w v)_h+\mu^2(v,v) + s(v,v).
\end{equation}
Then, $\3bar\cdot\3bar$ defines a norm in the linear space $V_h^0$.
\end{lemma}

\begin{proof} It is easily seen that $\3barv\3bar$ defines a semi norm on the finite element
space $V_h^0$ when we write  the term $2\mu (\kappa \nabla_w v,\nabla_w v)_h$ as $2\mu (\kappa^{\frac{1}{2}} \nabla_w v,\kappa^{\frac{1}{2}}\nabla_w v)_h$. 
We shall only verify the positivity property for
$\3bar\cdot\3bar$. To this end, assume that $\3barv\3bar=0$ for some
$v\in V_h^0$. It follows from (\ref{3barnorm}) that
$E_{w}v=0$  on each element  $T$, $ \kappa \nabla
v_0 \cdot \textbf{n}= v_g$ and
$ v_0=v_b$ on each $\partial T$. Thus, for any $\varphi \in P_{k-2}(T)$,
from  (\ref{A.002}), we 
obtain
\begin{equation*}
\begin{split}
0=&(E_{ w}v,\varphi)_T\\
 =&(Ev_0,\varphi)_T+\langle
v_0-v_b, \kappa\nabla\varphi\cdot \textbf{n} \rangle_{\partial T}-\langle
 \kappa \nabla v_0\cdot \textbf{n}-v_{g },\varphi\rangle_{\partial T}\\
 =&(Ev_0,\varphi)_T,\\
 \end{split}
\end{equation*} 
 which implies that $E v_0 =0$ on
each element $T$. Thus, from integration by parts, we obtain
\begin{equation*}
 \begin{split}
  0&=
\sum_{T\in {\cal T}_h}(Ev_0, v_0)_T\\&=\sum_{T\in {\cal T}_h}-(\kappa\nabla v_0,\nabla v_0)_T+\langle \kappa\nabla v_0\cdot \bn,v_0\rangle_{\partial T}\\
&=\sum_{T\in {\cal T}_h}-(\kappa\nabla v_0,\nabla v_0)_T+\langle \kappa\nabla v_0\cdot \bn-v_g,v_0\rangle_{\partial T}\\
&=\sum_{T\in {\cal T}_h}-(\kappa\nabla v_0,\nabla v_0)_T,
 \end{split}
\end{equation*}
where we have used the fact that the sum for the terms associated
with $v_g  $ vanishes  (note that 
 $v_g$ vanishes on $\partial T\cap\partial\Omega$).
 This implies that $\nabla v_0=0$ on
each element $T$.  Thus,  $v_0$ is a constant on
each element $T$, which, together with the fact that $ v_0=v_b$ on each
$\partial T$,   indicates that $v_0$ is continuous over the whole
domain $\Omega$. Thus, we obtain that $v_0=C$ in $\Omega$.    Together with the facts  that  $ v_0=v_b$ on
$\partial T$ and $v_b|_{\partial T\cap \partial\Omega}=0$,  we obtain $v_0=0$ in $\Omega$, which, 
combining with the facts that $ v_0=v_b$ on each
$\partial T$ and  $v_b|_{\partial T\cap \partial\Omega}=0$  indicates that $v_b=0$ in $\Omega$. Furthermore,
 the facts  that    $ \kappa \nabla
v_0 \cdot \textbf{n}= v_g$  on  each $\partial T$ and  $v_g|_{\partial T\cap \partial\Omega}=0$   yield   $v_g=0$ in $\Omega$. Thus, $v=0$ in $\Omega$.
This completes the proof of the
lemma.
\end{proof}

\begin{lemma}\label{Lemma4.2} The weak Galerkin algorithm (\ref{2.7}) has a
unique solution.
\end{lemma}

\begin{proof} Let $u_h^{(1)}$ and $u_h^{(2)}$ be two different
solutions of  the weak Galerkin algorithm (\ref{2.7}). It is clear
that the difference  $e_h=u_h^{(1)}-u_h^{(2)}$ is a finite element
function in $V_h^0$ satisfying
\begin{equation}\label{2.10}
 (E_{w} e_h,E_{w}v)_h+2\mu (\kappa \nabla_w e_h,\nabla_w v)_h+\mu^2(e_h,v)  + s(e_h,v)=0, \quad \forall v \in V_h^0.
\end{equation}
By setting  $v=e_h$ in  (\ref{2.10}), we obtain
$$
 (E_{w} e_h,E_{w}e_h)_h+2\mu (\kappa \nabla_w e_h,\nabla_w e_h)_h+\mu^2(e_h,e_h) + s(e_h,e_h)=0.
$$
From Lemma \ref{Lemma4.1}, we get $e_h\equiv 0$, i.e.,
$u_h^{(1)}=u_h^{(2)}$. This completes the proof.
\end{proof}
 
\section{$L^2$ Projections and Their
Properties}\label{Section:L2projection}
The goal of this section is to establish some technical results for
the $L^2$ projections. These results are valuable in the error
analysis for the WG finite element method.

For any $T\in\T_h$, let $\varphi$ be a regular function in $H^1(T)$.
The following trace inequality holds true \cite{wy3655}:
\begin{equation}\label{trace-inequality}
\|\varphi\|_e^2 \leq C
(h_T^{-1}\|\varphi\|_T^2+h_T\|\nabla\varphi\|_T^2).
\end{equation}
If $\varphi$ is a polynomial on the element $T\in \T_h$, then we
have from the inverse inequality (see also \cite{wy3655}) that
\begin{equation}\label{x}
\|\varphi\|_e^2 \leq C h_T^{-1}\|\varphi\|_T^2.
\end{equation}
Here $e$ is an edge or flat face on the boundary of $T$.

On each element $T\in {\cal T}_h$, define ${\cal
Q}_h$   the local $L^2$ projection onto $P_{k-2}(T)$ and  ${\cal
Q}_{1}$  the local $L^2$ projection onto $P_{k-1}(T)$.

\begin{lemma}\label{Lemma5.1}  The
$L^2$ projections $Q_h$, ${\cal
Q}_{1}$  and ${\cal Q}_h$ satisfy the following
commutative properties:
\begin{equation}\label{l}
E_{w}(Q_h w)={\cal Q}_h(E w),\qquad \forall w\in H_\kappa^2(T),
\end{equation}
\begin{equation}\label{l-2}
\nabla_{w}(Q_h w)={\cal Q}_1( \nabla w),\qquad \forall w\in H^1(T).
\end{equation}
\end{lemma}
\begin{proof} As to (\ref{l}), for any  $\varphi\in P_{k-2}(T)$ and $w\in H_\kappa^2(T)$, from the definition (\ref{2.4}) of $E_{ w}$
and the usual integration by parts, we have
\begin{equation*}
\begin{split}
(E_{ w}(Q_h w),\varphi)_T&=(Q_0
w,E \varphi)_T-\langle  Q_b w,\kappa \nabla \varphi\cdot
\textbf{n} \rangle_{\partial T}+
\langle   Q_g(\kappa \nabla w \cdot \textbf{n}) ,\varphi \rangle_{\partial T}\\
& =( 
w,E \varphi)_T -\langle  w, \kappa  \nabla \varphi\cdot
\textbf{n} \rangle_{\partial T}+
\langle \kappa \nabla w \cdot \textbf{n},\varphi  \rangle_{\partial T}\\
& =( \varphi
,E w)_T  \\
& =( \varphi
,{\cal Q}_h(Ew))_T.  
\end{split}
\end{equation*}

As to (\ref{l-2}), for any  $\boldsymbol{\psi}\in [P_{k-1}(T)]^d$ and $w\in H^1(T)$, from the definition  (\ref{2.4-2})  of $ \nabla_{ w}$
and the usual integration by parts, we have
\begin{equation*}
\begin{split}
( \nabla_{ w}(Q_h w),\boldsymbol{\psi})_T&=- (Q_0
w,\nabla \cdot \boldsymbol{\psi})_T+\langle  Q_b w,\boldsymbol{\psi} \cdot
\textbf{n} \rangle_{\partial T} \\
& =- ( 
w,\nabla \cdot \boldsymbol{\psi})_T+\langle   w,\boldsymbol{\psi} \cdot
\textbf{n} \rangle_{\partial T}\\
& =(  \boldsymbol{\psi}
, \nabla w)_T  \\
& =( \boldsymbol{\psi} 
,{\cal Q}_1( \nabla w))_T,   
\end{split}
\end{equation*}
which completes the proof.
\end{proof}

 The following lemma provides some approximation properties for the
projection operators $Q_h$, ${\cal Q}_1$ and ${\cal Q}_h$.

\begin{lemma}\label{Lemma5.2}\cite{mwy0927, wy3655}  Let ${\cal T}_h$ be a
finite element partition of  $\Omega$ satisfying the shape
regularity assumption as defined in \cite{wy3655}. Then, for any
$0\leq s\leq 2$ and $2\leq m\leq k$, there exists a constant $C$ such that
the following estimates hold true:
\begin{equation}\label{3.2}
\sum_{T\in {\cal T}_h}h_T^{2s}\|u-Q_0u\|^2_{s,T}\leq Ch^{2(m+1)}\|u\|_{m+1}^2,
\end{equation}
\begin{equation}\label{3.3}
\sum_{T\in {\cal
T}_h}h_T^{2s}\|Eu-{\cal
Q}_hEu\|^2_{s,T}\leq Ch^{2(m-1)}\|u\|_{m+1}^2,
\end{equation}
\begin{equation}\label{3.2-2}
\sum_{T\in {\cal
T}_h} h_T^{2s} \| \kappa\nabla u -    {\cal Q}_1 (\kappa  \nabla u )\|^2_{s,T}\leq Ch^{2m}\|u\|_{m+1}^2.
\end{equation}
\end{lemma}

\begin{lemma}\label{Lemma5.3} Let $2\leq m\leq k$ and $u\in
H^{\max\{m+1,4\}}(\Omega)$. There exists a constant $C$ such that
the following estimates hold true:
\begin{equation}\label{3.5}
\Big(\sum_{T\in {\cal
T}_h} h_T\|Eu-{\cal
Q}_h(Eu)\|_{\partial T}^2\Big)^{\frac{1}{2}}\leq
Ch^{m-1}\|u\|_{m+1},
\end{equation}
\begin{equation}\label{3.6}
%\begin{split}
\Big(\sum_{T\in {\cal
T}_h} h_T^3\|\kappa \nabla(E
 u-{\cal Q}_h(Eu)) \cdot \textbf{n}\|_{\partial T}^2\Big)^{\frac{1}{2}}\\
\leq Ch^{m-1}(\|u\|_{m+1}+h\delta_{m,2}\|u\|_4),
%\end{split}
\end{equation}
\begin{equation}\label{3.7}
\Big(\sum_{T\in {\cal T}_h}h_T^{-1}\| \kappa\nabla (Q_0u)\cdot \textbf{n}-Q_g(\kappa \nabla
u\cdot \textbf{n})\|_{\partial T}^2\Big)^{\frac{1}{2}}\leq Ch^{m-1}\|u\|_{m+1},
\end{equation}
\begin{equation}\label{3.8}
\Big(\sum_{T\in {\cal T}_h}h_T^{-3}\| 
  Q_0u - Q_bu\|_{\partial T}^2\Big)^{\frac{1}{2}}\leq
 Ch^{m-1}\|u\|_{m+1},
\end{equation}
\begin{equation}\label{3.8-2}
\Big(\sum_{T\in {\cal T}_h} h_T^3\|(\kappa\nabla u -    {\cal Q}_1 (\kappa  \nabla u ))\cdot \bn \|_{\partial T}^2\Big)
^{\frac{1}{2}}\leq
 Ch^{m+1}\|u\|_{m+1}.
\end{equation}
Here $\delta_{i,j}$ is the usual Kronecker's delta with value $1$
when $i=j$ and value $0$ otherwise.
\end{lemma}

\begin{proof} To prove (\ref{3.5}), by the trace inequality (\ref{trace-inequality}) and the
estimate (\ref{3.3}), we get
\begin{equation*}
\begin{split}
&\sum_{T\in {\cal T}_h} h_T\|Eu-{\cal Q}_h(Eu)\|_{\partial T}^2\\
\leq & C\sum_{T\in {\cal T}_h}  \|Eu-{\cal Q}_h(Eu)\|_{T}^2+h_T^2|Eu-{\cal Q}_h(Eu)|_{1,T}^2 \\
\leq & Ch^{2m-2}\|u\|^2_{m+1}.
\end{split}
\end{equation*}

As to (\ref{3.6}), by the trace inequality (\ref{trace-inequality})
and the estimate (\ref{3.3}), we obtain
\begin{equation*}
\begin{split}
&\sum_{T\in {\cal
T}_h} h_T^3\|\kappa \nabla(E
 u-{\cal Q}_h(Eu)) \cdot \textbf{n}\|_{\partial T}^2\\
 \leq & C\sum_{T\in {\cal
T}_h} h_T^2\|\nabla  (Eu-{\cal
Q}_h(Eu))\|_{T}^2
+h_T^4| \nabla (Eu-{\cal Q}_h(Eu))|_{1,T}^2\Big)\\
\leq & Ch^{2m-2}\big(\|u\|^2_{m+1}+h^2\delta_{m,2}\|u\|_4^2).
\end{split}
\end{equation*}

As to (\ref{3.7}), by  the trace inequality (\ref{trace-inequality})
and the estimate (\ref{3.2}), we have

\begin{equation*}
\begin{split}
& \sum_{T\in {\cal T}_h}h_T^{-1}\| \kappa\nabla (Q_0u)\cdot \textbf{n}-Q_g(\kappa \nabla
u\cdot \textbf{n})\|_{\partial T}^2 \\
\leq& \sum_{T\in {\cal T}_h}h_T^{-1}\| \kappa\nabla (Q_0u)\cdot \textbf{n}- \kappa \nabla
u\cdot \textbf{n} \|_{\partial T}^2 \\
\leq& C\sum_{T\in {\cal T}_h}h_T^{-1}\|  \nabla (Q_0u)  -  \nabla
u  \|_{\partial T}^2 \\
\leq& C\sum_{T\in {\cal T}_h} h_T^{-2}\|\nabla Q_0u-\nabla u\|_{ T}^2+|\nabla Q_0u-\nabla u|_{1,T}^2 \\
\leq&  Ch^{2m-2}\|u\|^2_{m+1}.
\end{split}
\end{equation*}

As to (\ref{3.8}), by the trace inequality
(\ref{trace-inequality}) and the estimate (\ref{3.2}), we have
\begin{equation*}
\begin{split}
&\sum_{T\in {\cal T}_h}h_T^{-3}\| Q_0u - Q_bu\|_{\partial T}^2\\
\leq & \sum_{T\in {\cal T}_h}h_T^{-3}\| Q_0u -  u\|_{\partial T}^2\\
 \leq& C\sum_{T\in {\cal T}_h} h_T^{-4}\|Q_0u-u\|_{T}^2+h_T^{-2}\|\nabla(Q_0u-u)\|_{T}^2 \\
\leq&  Ch^{2m-2}\|u\|^2_{m+1}.
\end{split}
\end{equation*}

Finally, as to  (\ref{3.8-2}), by the trace inequality
(\ref{trace-inequality}) and the estimate (\ref{3.2-2}), we have
\begin{equation*}
\begin{split}
& \sum_{T\in {\cal T}_h} h_T^3\|(\kappa\nabla u -    {\cal Q}_1 (\kappa  \nabla u ))\cdot \bn \|_{\partial T}^2 \\
\leq & \sum_{T\in {\cal T}_h} h_T^2\| \kappa\nabla u -    {\cal Q}_1 (\kappa  \nabla u )  \|^2_{  T}+h_T^4 | \kappa\nabla u -    {\cal Q}_1 (\kappa  \nabla u )  |_{ 1, T}^2
  \\
\leq&  Ch^{2m+2}\|u\|^2_{m+1}.
\end{split}
\end{equation*}
This completes the proof of the lemma.
\end{proof}

\section{An Error Equation}\label{Section:error-equation}
Let $u$ and $u_h=\{u_0,u_b, u_g\} \in V_h$ be the solutions of
(\ref{0.1}) and (\ref{2.7}), respectively. Denote by
\begin{equation}\label{error-term}
e_h=Q_hu-u_h
\end{equation}
the error function between the $L^2$ projection of the exact
solution $u$ and its weak Galerkin finite element approximation
$u_h$. An error equation refers to some identity that the error
function $e_h$ must satisfy. The goal of this section is to derive
an error equation for $e_h$.

\begin{lemma}\label{Lemma6.1} The error function
$e_h$ as defined by (\ref{error-term}) is a finite element function
in $V_h^0$ and satisfies the following equation
\begin{equation}\label{4.1}
(E_{w} e_h,E_{w}v)_h+2\mu (\kappa \nabla_w e_h,\nabla_w v)_h+\mu^2(e_h,v)_h+s(e_h,v)=\phi_u(v),\qquad
\forall v\in V_h^0,
\end{equation}
where
\begin{equation}\label{phiu}
\begin{split}
\phi_u(v)=
 &-\sum_{T\in {\cal T}_h}\langle  \kappa  \nabla (Eu- {\cal
Q}_h(Eu)  ) \cdot \textbf{n},
 v_0 -v_b \rangle_{\partial T} \\
&+\sum_{T\in {\cal T}_h} \langle  \kappa \nabla v_0\cdot \textbf{n}-v_g, Eu-{\cal Q}_hE u\rangle_{\partial
T}\\
&+  \sum_{T\in {\cal T}_h} 2\mu  \langle v_0-v_b,   (\kappa\nabla u -    {\cal Q}_1 (\kappa  \nabla u ))\cdot \bn   \rangle_{\partial T}  +s(Q_hu,v).\end{split}
\end{equation}
\end{lemma}

\begin{proof} Using (\ref{A.002}) with
$\varphi=E_w(Q_h u )$, from (\ref{l}),  we obtain
\begin{equation*}
\begin{split}
(E_wv,E_{ w}(Q_hu))_T=&(Ev_0,{\cal
Q}_h(Eu))_T+\langle v_0-v_b, \kappa\nabla ({\cal
Q}_h(Eu))
\cdot \textbf{n} \rangle_{\partial T}\\
&-\langle\kappa \nabla v_0\cdot \textbf{n}-v_{g},{\cal Q}_hE u\rangle_{\partial T}\\
=&(Ev_0, Eu)_T+\langle
v_0-v_b,\kappa \nabla ({\cal Q}_h(Eu))\cdot
\textbf{n} \rangle_{\partial T}\\
&-\langle\kappa \nabla v_0\cdot
\textbf{n}  -v_{g },{\cal Q}_hE
u\rangle_{\partial T},
\end{split}
\end{equation*}
which implies that
\begin{equation}\label{4.2}
\begin{split}
(Ev_0, Eu)_T=&
(E_{ w}(Q_hu),E_{ w}v)_T-\langle v_0-v_b,
\kappa \nabla({\cal Q}_h(Eu))\cdot \textbf{n} \rangle_{\partial T}\\
&+\langle \kappa \nabla v_0\cdot
\textbf{n}-v_{g },{\cal Q}_hE
u\rangle_{\partial T}.
\end{split}
\end{equation}
 
 Next, it follows from the
integration by parts that
\begin{equation}\label{part1}
 \begin{split}
\sum_{T\in {\cal T}_h}(Eu, Ev_0)_T=&
\sum_{T\in {\cal T}_h} (E^2u,v_0)_T-\langle  \kappa \nabla (Eu) \cdot \textbf{n},
 v_0 \rangle_{\partial T}
+\langle  \kappa \nabla v_0\cdot \textbf{n}, Eu\rangle_{\partial
T} .
 \end{split}
\end{equation}

  Using  (\ref{2.4-3}) with  $\boldsymbol{\psi}=\nabla_w(Q_hu)$, from  (\ref{l-2})  and  the  integration by parts, we have
\begin{equation}\label{part2}
 \begin{split}
& \sum_{T\in {\cal T}_h} 2\mu (\kappa \nabla_w(Q_hu), \nabla_w v)_T\\
 =&  \sum_{T\in {\cal T}_h}2\mu (\kappa\nabla v_0,{\cal Q}_1  \nabla u )_T -2\mu \langle v_0-v_b , {\cal Q}_1 \kappa \nabla u \cdot \bn\rangle_{\partial T}\\
 =&  \sum_{T\in {\cal T}_h}2\mu (\nabla v_0, \kappa \nabla u )_T -2\mu  \langle v_0-v_b , {\cal Q}_1 \kappa\nabla u \cdot \bn\rangle_{\partial T}\\
=&  \sum_{T\in {\cal T}_h}-2\mu (v_0, \nabla \cdot( \kappa \nabla u ))_T + 2\mu  \langle  v_0,  \kappa \nabla u \cdot  \bn  \rangle_{\partial T} -2\mu \langle v_0-v_b , {\cal Q}_1 \kappa  
\nabla u \cdot \bn\rangle_{\partial T}\\
=& \sum_{T\in {\cal T}_h} -2\mu (v_0,E u)_T + 2\mu  \langle v_0,   (\kappa\nabla u -    {\cal Q}_1 (\kappa  \nabla u ))\cdot \bn   \rangle_{\partial T}  ,
 \end{split}
\end{equation}
where we have used the fact that the sum for the terms associated
with $v_{b}$   vanishes  (note that 
$v_{b}$  vanishes on $\partial T\cap\partial\Omega$).

From the definition of the projection, we get
\begin{equation}\label{part3}
\sum_{T\in {\cal T}_h}  \mu^2 (Q_hu,v)_T=\sum_{T\in {\cal T}_h}  \mu^2 ( u_0,v_0)_T.
\end{equation}

Adding (\ref{part1})-(\ref{part3}) together and  using the identity that 
$$
\sum_{T\in {\cal T}_h} (E^2
u,v_0)_T-2\mu (Eu,v_0)_T+\mu^2(u_0,v_0)_T = (f,v_0),
 $$ we obtain
\begin{equation*}
\begin{split}
&\sum_{T\in {\cal T}_h} (Eu, Ev_0)_T+ 2\mu \kappa (\nabla_w(Q_hu), \nabla_w v)_T+\mu^2 (Q_hu,v)_T  \\
=& (f,v_0)-\sum_{T\in {\cal T}_h}\Big( \langle  \kappa \nabla (Eu) \cdot \textbf{n},
 v_0 -v_b \rangle_{\partial T} +  \langle  \kappa \nabla v_0\cdot \textbf{n}-v_g, Eu\rangle_{\partial
T}\\
&+  2\mu  \langle v_0-v_b,   (\kappa\nabla u -    {\cal Q}_1 (\kappa  \nabla u ))\cdot \bn   \rangle_{\partial T}\Big),
\end{split}
\end{equation*}
where we have used the fact that the sum for the terms associated
with $v_{b}$ and $v_g  $ vanishes  (note that both
$v_{b}$ and $v_g$ vanish on $\partial T\cap\partial\Omega$). Combining the
above equation with (\ref{4.2}) and adding $s(Q_hu,v)$ to both sides of the   equation  yield 
\begin{equation*}\label{4.3}
\begin{split}
&\sum_{T\in {\cal T}_h}(E_{ w}(Q_hu),E_{ w}v)_T+ 2\mu \kappa (\nabla_w(Q_hu), \nabla_w v)_T+\mu^2 (Q_hu,v)_T \\
&+s(Q_hu,v)\\
%=& (f,v_0)-\sum_{T\in {\cal T}_h}\Big(\langle  \kappa \nabla (Eu) \cdot \textbf{n},
 %v_0 -v_b \rangle_{\partial T} +   \langle  \kappa \nabla v_0\cdot \textbf{n}-v_g, Eu\rangle_{\partial
%T}\\
%&+  2\mu  \langle v_0-v_b,   (\kappa\nabla u -    {\cal Q}_1 (\kappa  \nabla u ))\cdot \bn   \rangle_{\partial T} \\
% &+\langle v_0-v_b, \kappa\nabla ({\cal
%Q}_h(Eu))
%\cdot \textbf{n} \rangle_{\partial T} -\langle\kappa \nabla v_0\cdot \textbf{n}-v_{g},{\cal Q}_hE u\rangle_{\partial T}\Big)\\
=& (f,v_0)-\sum_{T\in {\cal T}_h}\Big( \langle  \kappa  \nabla (Eu- {\cal
Q}_h(Eu)  ) \cdot \textbf{n},
 v_0 -v_b \rangle_{\partial T} \\
&+ \langle  \kappa \nabla v_0\cdot \textbf{n}-v_g, Eu-{\cal Q}_hE u\rangle_{\partial
T}\\
&+  2\mu  \langle v_0-v_b,   (\kappa\nabla u -    {\cal Q}_1 (\kappa  \nabla u ))\cdot \bn   \rangle_{\partial T}\Big) +s(Q_hu,v),
\end{split}
\end{equation*}
%Adding $s(Q_hu,v)$ to both sides of the above equation gives
%\begin{equation}
%\begin{split}
 %&\sum_{T\in {\cal T}_h}  (E_{ w}(Q_hu),E_{ w}v)_T+ 2\mu \kappa (\nabla_w(Q_hu), \nabla_w v)_T+\mu^2 (Q_hu,v)_T \\
%&+s(Q_hu,v)\\
%=& (f,v_0)-\sum_{T\in {\cal T}_h}\Big( \langle  \kappa  \nabla (Eu- {\cal
%Q}_h(Eu)  ) \cdot \textbf{n},
 %v_0 -v_b \rangle_{\partial T} \\
%&+  \langle  \kappa \nabla v_0\cdot \textbf{n}-v_g, Eu-{\cal Q}_hE u\rangle_{\partial T}\\
%&+  2\mu  \langle v_0-v_b,   (\kappa\nabla u -    {\cal Q}_1 (\kappa  \nabla u ))\cdot \bn   \rangle_{\partial T} \Big) +s(Q_hu,v).
%\end{split}
%\end{equation}
which, subtracting (\ref{2.7}),   completes the proof.
\end{proof}

\section{Error Estimates in $H^2$}\label{Section:H2error}
The goal of this section is to derive some error estimates for the
solution of weak Galerkin algorithm (\ref{2.7}). 
The following result is an estimate for the error function $e_h$ in
the trip-bar norm which is essentially an $H^2$-equivalent norm in
$V_h^0$.

\begin{theorem}\label{Theorem6.6} Let $u_h\in V_h$ be the weak Galerkin finite
element solution arising from (\ref{2.7}) with finite elements of
order $k\geq 2$. Assume that the exact solution $u$ of (\ref{0.1})
is sufficiently regular such that $u\in H^{\max\{k+1,4\}}(\Omega)$.
Then, there exists a constant $C$ such that
\begin{equation}\label{4}
\3baru_h-Q_hu\3bar\leq
Ch^{k-1}\Big(\|u\|_{k+1}+h\delta_{k,2}\|u\|_4\Big).
\end{equation}
In other words, we have an optimal order of convergence in the $H^2$-equivalent  
norm.
\end{theorem}

\begin{proof} Letting  $v=e_h$ in the error equation (\ref{4.1}) gives rise to
\begin{equation}\label{4.4}
\begin{split}
\3bar e_h\3bar^2=&-\sum_{T\in {\cal T}_h}\langle  \kappa  \nabla (Eu- {\cal
Q}_h(Eu)  ) \cdot \textbf{n},
 e_0 -e_b \rangle_{\partial T} \\
&+\sum_{T\in {\cal T}_h} \langle  \kappa \nabla e_0\cdot \textbf{n}-e_g, Eu-{\cal Q}_hE u\rangle_{\partial
T}\\
&+ \sum_{T\in {\cal T}_h} 2\mu  \langle e_0-e_b,   (\kappa\nabla u -    {\cal Q}_1 (\kappa  \nabla u ))\cdot \bn   \rangle_{\partial T} \\
&+\sum_{T\in {\cal T}_h}  h_T^{-1}\langle  \kappa \nabla
Q_0u \cdot \textbf{n}- Q_g( \kappa \nabla u \cdot \textbf{n}), \kappa \nabla
e_0 \cdot \textbf{n}-
e_g\rangle_{\partial T}\\
&+\sum_{T\in {\cal T}_h} h_T^{-3}\langle Q_0u-Q_bu,  
e_0-e_b\rangle_{\partial T}.
\end{split}
\end{equation}

Each  term on the right-hand side of (\ref{4.4}) can be estimated as follows. For the first term of the right-hand side of (\ref{4.4}), using the Cauchy-Schwarz inequality, and the estimate (\ref{3.6}), one arrives at
\begin{equation*}\label{4.5}
\begin{split}
 &\Big|\sum_{T\in {\cal T}_h}\langle  \kappa \nabla (Eu-{\cal Q}_h Eu) \cdot \textbf{n},
 e_0-e_b \rangle_{\partial T}\Big|\\
 \leq&\Big(\sum_{T\in {\cal T}_h}h_T^3\| \kappa \nabla (Eu-{\cal Q}_h Eu) \cdot \textbf{n}\|^2_{\partial T}\Big)^{\frac{1}{2}}
\Big(\sum_{T\in {\cal T}_h}h_T^{-3}\| e_0-e_b \|^2_{\partial T}\Big)^{\frac{1}{2}}\\
  \leq& Ch^{k-1}(\|u\|_{k+1}+h\delta_{k,2}\|u\|_4)\3bar e_h\3bar.
\end{split}
\end{equation*}
For the second term of the right-hand side of (\ref{4.4}), using the Cauchy-Schwarz inequality, and 
the estimate  (\ref{3.5}), one arrives at
\begin{equation*}\label{4.6}
\begin{split}
 &\Big|\sum_{T\in {\cal T}_h}\langle  \kappa \nabla e_0\cdot \textbf{n}-e_g, Eu-{\cal Q}_h Eu\rangle_{\partial
T}\Big|\\
 \leq&\Big(\sum_{T\in {\cal T}_h}h_T^{-1}\|\kappa \nabla e_0\cdot \textbf{n}-e_g\|^2_{\partial T}\Big)^{\frac{1}{2}}
\Big(\sum_{T\in {\cal T}_h}h_T \|Eu-{\cal Q}_h Eu\|^2_{\partial T}\Big)^{\frac{1}{2}}\\
  \leq& Ch^{k-1} \|u\|_{k+1} \3bar e_h\3bar.
\end{split}
\end{equation*}
For the third term of the right-hand side of (\ref{4.4}), using the Cauchy-Schwarz inequality, and the estimate (\ref{3.8-2}), one arrives at
\begin{equation*}\label{4.6-2}
\begin{split}
 &\Big|\sum_{T\in {\cal T}_h} 2\mu  \langle e_0-e_b,   (\kappa\nabla u -    {\cal Q}_1 (\kappa  \nabla u ))\cdot \bn   \rangle_{\partial T}
\Big|\\
 \leq&\Big(\sum_{T\in {\cal T}_h}h_T^{-3}\|e_0-e_b\|^2_{\partial T}\Big)^{\frac{1}{2}}
\Big(\sum_{T\in {\cal T}_h}h_T ^3\| (\kappa\nabla u -    {\cal Q}_1 (\kappa  \nabla u ))\cdot \bn  \|^2_{\partial T}\Big)^{\frac{1}{2}}\\
  \leq& Ch^{k+1} \|u\|_{k+1} \3bar e_h\3bar.
\end{split}
\end{equation*}
For the fourth term of the right-hand side of (\ref{4.4}), using the Cauchy-Schwarz inequality, and the estimate  (\ref{3.7}), one arrives at
\begin{equation*}\label{4.7}
\begin{split}
 &\Big|\sum_{T\in {\cal T}_h} h_T^{-1}\langle  \kappa \nabla
Q_0u \cdot \textbf{n}- Q_g( \kappa \nabla u \cdot \textbf{n}), \kappa \nabla
e_0 \cdot \textbf{n}-
e_g\rangle_{\partial T}\Big|\\
   \leq& Ch^{k-1} \|u\|_{k+1} \3bar e_h\3bar.
\end{split}
\end{equation*}
For the last term of the right-hand side of (\ref{4.4}), using the Cauchy-Schwarz inequality, and the estimate  (\ref{3.8}), one arrives at
\begin{equation*}\label{4.8}
\begin{split}
 &\Big|\sum_{T\in {\cal T}_h} h_T^{-3}\langle Q_0u-Q_bu,  
e_0-e_b\rangle_{\partial T} \Big|\\
    \leq& Ch^{k-1} \|u\|_{k+1} \3bar e_h\3bar.
\end{split}
\end{equation*}

Substituting  all the above estimates  into (\ref{4.4}) yields
$$
\3bar e_h\3bar^2\leq Ch^{k-1}(\|u\|_{k+1}+h\delta_{k,2}\|u\|_4)\3bar e_h\3bar,
$$  
which implies (\ref{4}). This completes the proof of the theorem.
\end{proof}

\section{Error Estimates in $L^2$}\label{Section:L2error}
This section shall establish the estimates for the  components $e_0$, $e_b$ and $e_g$ of
the error function $e_h$ in the standard $L^2$ norm, respectively. To this end, we
consider the following dual problem:
\begin{equation}\label{5.1}
\begin{split}
F^2\w&=e_0, \qquad \text{in}\ \Omega,\\
\w&=0, \qquad \text{on}\ \partial\Omega,\\
 \kappa \nabla \w\cdot \textbf{n}&=0,  \qquad \text{on}\
\partial\Omega.
\end{split}
\end{equation}
Assume the above dual problem has the following regularity estimate
\begin{equation}\label{5.2}
\|\w\|_4\leq C\|e_0\|.
\end{equation}

\begin{theorem}\label{Theorem7.3} Let $u_h\in V_h$ be the solution of
the weak Galerkin algorithm (\ref{2.7}) with finite elements of
order $k\geq 2$. Let $t_0=\min\{k,3\}$. Assume that the exact
solution of (\ref{0.1}) is sufficiently regular such that  $u\in H^{\max\{k+1,4\}}(\Omega)$, and
the dual problem (\ref{5.1}) has the $H^4$ regularity. Then, there
exists a constant $C$ such that
\begin{equation}\label{5.3}
\|Q_0u-u_0\|\leq
Ch^{k+t_0-2}\Big(\|u\|_{k+1}+h\delta_{k,2}\|u\|_{4}\Big).
\end{equation}
In other words, we have a sub-optimal order of convergence for $k=2$
and optimal order of convergence for $k\geq 3$.
\end{theorem}

\begin{proof} By testing the first equation of (\ref{5.1}) against the error function
$e_0$ on each element and using the integration by parts, we obtain
\begin{equation*}
\begin{split}
\|e_0\|^2=&(F^2 \w,e_0)\\
=&\sum_{T\in {\cal
T}_h} \Big\{(E\w,Ee_0)_T-\langle E\w,
\kappa \nabla e_0\cdot \textbf{n} \rangle_{\partial T}+\langle  \kappa \nabla (E\w)\cdot \textbf{n},e_0\rangle_{\partial T}\\
&+(-2\mu E\w+\mu^2\w, e_0)_T\Big\}\\
=&\sum_{T\in {\cal
T}_h} \Big\{(E\w,Ee_0)_T-\langle E\w,
\kappa \nabla e_0\cdot \textbf{n} -e_g\rangle_{\partial T}+\langle  \kappa \nabla (E\w)\cdot \textbf{n},e_0-e_b\rangle_{\partial T}\\
&+(-2\mu E\w+\mu^2\w, e_0)_T\Big\},
\end{split}
\end{equation*}
where the added terms associated with $e_b$ and $e_{g}$ vanish due
to the cancelation for interior edges and the fact that $e_b$ and
$e_{g}$ have zero value on $\partial\Omega$. Using (\ref{4.2}) with
$\w$ and $e_h$ in the place of $u$ and $v$ respectively, we arrive
at
\begin{equation}\label{09-100}
\begin{split}
\|e_0\|^2 =& 
\sum_{T\in {\cal
T}_h} \Big\{ (E_{ w}(Q_h\w),E_{ w}e_h)_T-\langle e_0-e_b,
\kappa \nabla({\cal Q}_h(E\w))\cdot \textbf{n} \rangle_{\partial T}\\
&+\langle \kappa \nabla e_0\cdot
\textbf{n}-e_{g },{\cal Q}_hE
\w\rangle_{\partial T} \\
&-\langle E\w,
\kappa \nabla e_0\cdot \textbf{n} -e_g\rangle_{\partial T}+\langle  \kappa \nabla (E\w)\cdot \textbf{n},e_0-e_b\rangle_{\partial T}\\
&+(-2\mu E\w+\mu^2\w, e_0)_T\Big\}\\
=&\sum_{T\in {\cal
T}_h} \Big\{ (E_{ w}(Q_h\w),E_{ w}e_h)_T-\langle e_0-e_b,
\kappa \nabla({\cal Q}_h(E\w)-E\w)\cdot \textbf{n}  \rangle_{\partial T}\\
&+\langle \kappa \nabla e_0\cdot
\textbf{n}-e_{g },{\cal Q}_hE
\w- E\w\rangle_{\partial T}  +(-2\mu E\w+\mu^2\w, e_0)_T\Big\}.
\end{split}
\end{equation}

Next, it follows from the error equation (\ref{4.1})  with $v=Q_h\w$ that
\begin{equation}\label{09-101}
\begin{split}
&(E_{w} e_h,E_{w}Q_h\w)_h\\
=&-2\mu (\kappa \nabla_w e_h,\nabla_w Q_h\w)_h-\mu^2(e_h,Q_h\w)_h-s(e_h,Q_h\w)+\phi_u(Q_h\w).
\end{split}
\end{equation}

Substituting (\ref{09-101}) into (\ref{09-100}), combining by (\ref{2.4-2}) with $\boldsymbol{\psi}=\nabla_w Q_h\w$,
 and from  (\ref{l-2}), we obtain
\begin{equation}\label{5.4}
\begin{split}
& \|e_0\|^2 \\
=&-2\mu (\kappa \nabla_w e_h,\nabla_w Q_h\w)_h-\mu^2(e_h,Q_h\w)_h-s(e_h,Q_h\w) \\
&-\sum_{T\in {\cal T}_h}\Big\{\langle  \kappa  \nabla (Eu- {\cal
Q}_h(Eu)  ) \cdot \textbf{n},
 Q_0\w-Q_b\w \rangle_{\partial T} \\
&+  \langle  \kappa \nabla Q_0\w\cdot \textbf{n}-Q_g(\kappa \nabla \w \cdot \bn), Eu-{\cal Q}_hE u\rangle_{\partial
T}\\
&+  2\mu  \langle Q_0\w-Q_b\w,   (\kappa\nabla u -    {\cal Q}_1 (\kappa  \nabla u ))\cdot \bn   \rangle_{\partial T}  \\
&-\langle e_0-e_b,
\kappa \nabla({\cal Q}_h(E\w)-E\w)\cdot \textbf{n}  \rangle_{\partial T}\\
&+\langle \kappa \nabla e_0\cdot
\textbf{n}-e_{g },{\cal Q}_hE
\w- E\w\rangle_{\partial T}  +(-2\mu E\w+\mu^2\w, e_0)_T\Big\}+s(Q_hu,Q_h\w)\\
=&\sum_{T\in {\cal T}_h}\Big\{ ( \mu^2(\w-Q_0\w)-2\mu \nabla \cdot (\kappa \nabla \w -{\cal Q}_1 (\kappa \nabla \w)  ), e_0)_T \\
&-2\mu \kappa \langle e_b, {\cal Q}_1 (\nabla \w) \cdot \bn \rangle_{\partial T} - \langle  \kappa  \nabla (Eu- {\cal
Q}_h(Eu)  ) \cdot \textbf{n},
 Q_0\w-Q_b\w \rangle_{\partial T} \\
&+  \langle  \kappa \nabla Q_0\w\cdot \textbf{n}-Q_g(\kappa \nabla \w \cdot \bn), Eu-{\cal Q}_hE u\rangle_{\partial
T}\\
&+  2\mu  \langle Q_0\w-Q_b\w,   (\kappa\nabla u -    {\cal Q}_1 (\kappa  \nabla u ))\cdot \bn   \rangle_{\partial T} \\
& -\langle e_0-e_b,
\kappa \nabla({\cal Q}_h(E\w)-E\w)\cdot \textbf{n}  \rangle_{\partial T}\\
&+\langle \kappa \nabla e_0\cdot
\textbf{n}-e_{g },{\cal Q}_hE
\w- E\w\rangle_{\partial T}   \Big\}-s(e_h,Q_h\w)+s(Q_hu,Q_h\w)\\
=&\sum_{T\in {\cal T}_h}\Big\{ ( \mu^2(\w-Q_0\w)-2\mu \nabla \cdot (\kappa \nabla \w -{\cal Q}_1 (\kappa \nabla \w)  ), e_0)_T \\
&- \langle  \kappa  \nabla (Eu- {\cal
Q}_h(Eu)  ) \cdot \textbf{n},
 Q_0\w-Q_b\w \rangle_{\partial T} \\
&+  \langle  \kappa \nabla Q_0\w\cdot \textbf{n}-Q_g(\kappa \nabla \w \cdot \bn), Eu-{\cal Q}_hE u\rangle_{\partial
T}\\
&+  2\mu  \langle Q_0\w-Q_b\w,   (\kappa\nabla u -    {\cal Q}_1 (\kappa  \nabla u ))\cdot \bn   \rangle_{\partial T} \\
& -\langle e_0-e_b,
\kappa \nabla({\cal Q}_h(E\w)-E\w)\cdot \textbf{n}  \rangle_{\partial T}\\
&+\langle \kappa \nabla e_0\cdot
\textbf{n}-e_{g },{\cal Q}_hE
\w- E\w\rangle_{\partial T}   \Big\}-s(e_h,Q_h\w)+s(Q_hu,Q_h\w),
\end{split}
\end{equation}
where we use  the third equation of (\ref{5.1}) to obtain 
 $$\sum_{T\in {\cal T}_h}2\mu \kappa \langle e_b, {\cal Q}_1 (\nabla \w) \cdot \bn \rangle_{\partial T}
=2\mu \kappa \langle e_b, {\cal Q}_1 (\nabla \w) \cdot \bn \rangle_{\partial  \Omega}
=2\mu \langle e_b,   \kappa \nabla \w  \cdot \bn \rangle_{\partial  \Omega}=0. $$
 
Each of these terms on the right-hand side of (\ref{5.4}) can be bounded as follows. 
Note that $t_0=\min\{3,k\}\leq 3$. For the first term of the right-hand side of (\ref{5.4}), from the Cauchy-Schwarz inequality, (\ref{5.2}),  the estimates (\ref{3.2}) and (\ref{3.2-2}), we have
\begin{equation*}\label{5.5-2} 
\begin{split}
&\Big|\sum_{T\in {\cal T}_h} ( \mu^2(\w-Q_0\w)-2\mu \nabla \cdot (\kappa \nabla \w -{\cal Q}_1 (\kappa \nabla \w)  ), e_0)_T \Big|\\
\leq &  \Big(\sum_{T\in {\cal T}_h} \|\w-Q_0\w\|_T^2  \Big)^{\frac{1}{2}}  \Big(\sum_{T\in {\cal T}_h} \|e_0\|_T^2  \Big)^{\frac{1}{2}}\\
&+
 \Big(\sum_{T\in {\cal T}_h} \|\nabla \cdot (\kappa \nabla \w -{\cal Q}_1 (\kappa \nabla \w) \|_T^2  \Big)^{\frac{1}{2}}  \Big(\sum_{T\in {\cal T}_h} \|e_0\|_T^2  \Big)^{\frac{1}{2}}\\
\leq & Ch^{k+1}\|\w\|_{k+1}\|e_0\|+Ch^{k-1}\|\w\|_{k+1}\|e_0\|\\
\leq & Ch^{k-1}\|\w\|_{k+1}\|e_0\|\\
\leq & Ch \|\w\|_{3}\|e_0\|\\
\leq & Ch \|\w\|_{4}\|e_0\|\\
\leq & Ch  \|e_0\|^2.
\end{split}
\end{equation*}
For the second term  of the right-hand side of (\ref{5.4}), it follows from the Cauchy-Schwarz inequality and the estimates (\ref{3.6})   and (\ref{3.8})  that
\begin{equation*}\label{5.5}
\begin{split}
 &\Big|\sum_{T\in {\cal T}_h}2\mu\langle  \kappa \nabla (Eu-{\cal Q}_h Eu) \cdot \textbf{n},
Q_0\w-Q_b\w \rangle_{\partial T}\Big|\\ 
\leq & \Big(\sum_{T\in {\cal T}_h}h_T^3\|  \kappa \nabla (Eu-{\cal Q}_h Eu) \cdot \textbf{n}\|^2_{\partial T}\Big)^{\frac{1}{2}}
\Big(\sum_{T\in {\cal T}_h}h_T^{-3}\| Q_0\w-Q_b\w  \|^2_{\partial T}\Big)^{\frac{1}{2}} \\
\leq &  Ch^{k-1}(\|u\|_{k+1}+h\delta_{k,2}\|u\|_4) h^{t_0-1}\|\w\|_{t_0+1}\\
\leq &  Ch^{k+t_0-2}(\|u\|_{k+1}+h\delta_{k,2}\|u\|_4)\|\w\|_{4}.\\
\end{split}
\end{equation*} 
For the third term  of the right-hand side of (\ref{5.4}), it follows from the Cauchy-Schwarz inequality and the estimates (\ref{3.5}) and (\ref{3.7})  that
\begin{equation*}\label{5.6}
\begin{split}
 &\Big|\sum_{T\in {\cal T}_h}\langle  \kappa \nabla Q_0\w\cdot \textbf{n}-Q_g(\kappa \nabla  \w\cdot \textbf{n}), Eu-{\cal Q}_h Eu\rangle_{\partial
T} \Big|\\ 
\leq & \Big(\sum_{T\in {\cal T}_h}h_T^{-1}\|   \kappa \nabla Q_0\w\cdot \textbf{n}-Q_g(\kappa \nabla  \w\cdot \textbf{n})\|^2_{\partial T}\Big)^{\frac{1}{2}}
\Big(\sum_{T\in {\cal T}_h}h_T\| Eu-{\cal Q}_h Eu \|^2_{\partial T}\Big)^{\frac{1}{2}} \\
 \leq &  Ch^{t_0-1}\|\w\|_{t_0+1}h^{k-1} \|u\|_{k+1} \\
\leq &  Ch^{k+t_0-2} \|u\|_{k+1} \|\w\|_{4}.\\
\end{split}
\end{equation*} 
For the fourth term  of the right-hand side of (\ref{5.4}), it follows from  the Cauchy-Schwarz inequality and  the estimates (\ref{3.8}) and (\ref{3.8-2}) that
\begin{equation*}\label{5.6-2} 
\begin{split}
 &\Big| 2\mu  \langle Q_0\w-Q_b\w,   (\kappa\nabla u -    {\cal Q}_1 (\kappa  \nabla u ))\cdot \bn   \rangle_{\partial T} \Big|\\ 
\leq & \Big(\sum_{T\in {\cal T}_h}h_T^{-3}\| Q_0\w-Q_b\w\|^2_{\partial T}\Big)^{\frac{1}{2}}
\Big(\sum_{T\in {\cal T}_h}h_T^3\|  (\kappa\nabla u -    {\cal Q}_1 (\kappa  \nabla u ))\cdot \bn  \|^2_{\partial T}\Big)^{\frac{1}{2}} \\
 \leq &  Ch^{t_0-1}\|\w\|_{t_0+1}h^{k+1} \|u\|_{k+1} \\
\leq &  Ch^{k+t_0} \|u\|_{k+1} \|\w\|_{4}.\\
\end{split}
\end{equation*}
For the fifth term  of the right-hand side of (\ref{5.4}), it follows from the Cauchy-Schwarz inequality and the estimate (\ref{3.6})  that
\begin{equation*}\label{5.8}
\begin{split}
 &\Big|\sum_{T\in {\cal T}_h} \langle  \kappa \nabla (E\w-{\cal Q}_h E\w) \cdot \textbf{n},
 e_0-e_b \rangle_{\partial T} \Big|\\ 
\leq & \Big(\sum_{T\in {\cal T}_h}h_T^3\|  \kappa \nabla (E\w-{\cal Q}_h E\w) \cdot \textbf{n}\|^2_{\partial T}\Big)^{\frac{1}{2}}
\Big(\sum_{T\in {\cal T}_h}h_T^{-3}\| e_0-e_b\|^2_{\partial T}\Big)^{\frac{1}{2}} \\
 \leq &  Ch^{t_0-1}(\|\w\|_{t_0+1}+h\delta_{t_0,2}\|\w\|_4) \3bar e_h\3bar \\
\leq &  Ch^{ t_0-1} \|\w\|_{4}\3bar e_h\3bar.
\end{split}
\end{equation*}
For the sixth term of the right-hand side of (\ref{5.4}), it follows from the Cauchy-Schwarz inequality and  the estimate (\ref{3.5})   that
\begin{equation*}\label{5.10}
\begin{split}
 &\Big|\sum_{T\in {\cal T}_h}\langle  \kappa \nabla e_0\cdot \textbf{n}-e_g,E\w-{\cal Q}_h E\w\rangle_{\partial
T}  \Big|\\ 
\leq & \Big(\sum_{T\in {\cal T}_h}h_T^{-1}\|  \kappa \nabla e_0\cdot \textbf{n}-e_g\|^2_{\partial T}\Big)^{\frac{1}{2}}
\Big(\sum_{T\in {\cal T}_h}h_T\| E\w-{\cal Q}_h E\w\|^2_{\partial T}\Big)^{\frac{1}{2}} \\
 \leq &  Ch^{t_0-1} \|\w\|_{t_0+1}  \3bar e_h\3bar \\
\leq &  Ch^{ t_0-1} \|\w\|_{4}\3bar e_h\3bar.
\end{split}
\end{equation*}
For the seventh term of the right-hand side of (\ref{5.4}), it follows from the Cauchy-Schwarz inequality and the estimates (\ref{3.7}) and (\ref{3.8})  that
\begin{equation*}\label{5.11}
\begin{split}
  & \Big|s(Q_h\w,e_h) \Big| \\
\leq & \Big|\sum_{T\in {\cal T}_h}h_T^{-1} \langle
\kappa\nabla Q_0\w\cdot \textbf{n}-Q_g( \kappa\nabla  \w\cdot \textbf{n}), \kappa\nabla e_0 \cdot \textbf{n}-e_g \rangle\Big|\\
 & + \Big|\sum_{T\in {\cal T}_h}h_T^{-3} \langle  Q_0\w-Q_b\w, 
 e_0 -e_b \rangle\Big|\\
\leq & \Big(\sum_{T\in {\cal T}_h}h_T^{-1}\|   \kappa\nabla Q_0\w\cdot \textbf{n}-Q_g( \kappa\nabla  \w\cdot \textbf{n})\|^2_{\partial T}\Big)
^{\frac{1}{2}} \Big(\sum_{T\in {\cal T}_h}h_T^{-1}\| \kappa\nabla e_0 \cdot \textbf{n}-e_g\|^2_{\partial T}\Big)^{\frac{1}{2}} \\
  & + \Big(\sum_{T\in {\cal T}_h}h_T^{-3}\|   Q_0\w-Q_b\w\|^2_{\partial T}\Big)^{\frac{1}{2}}
\Big(\sum_{T\in {\cal T}_h}h_T^{-3}\|  e_0 -e_b\|^2_{\partial T}\Big)^{\frac{1}{2}} \\
 \leq &  Ch^{t_0-1} \|\w\|_{t_0+1}  \3bar e_h\3bar \\
\leq &  Ch^{ t_0-1} \|\w\|_{4}\3bar e_h\3bar.
\end{split}
\end{equation*}
For the last term of the right-hand side of (\ref{5.4}), it follows from the Cauchy-Schwarz inequality and the estimates (\ref{3.7}) and (\ref{3.8})   that
\begin{equation*}\label{5.12}
\begin{split}
  &\qquad\Big|s(Q_hu,Q_h\w) \Big|\\ 
&\leq  \Big|\sum_{T\in {\cal T}_h}h_T^{-1} \langle \kappa\nabla Q_0u\cdot \textbf{n}-Q_g( \kappa\nabla  u\cdot \textbf{n}), 
\kappa\nabla Q_0\w\cdot \textbf{n}-Q_g( \kappa\nabla  \w\cdot \textbf{n})\rangle_{\partial T}\Big|\\
  &\qquad+ \Big|\sum_{T\in {\cal T}_h}h_T^{-3} \langle  Q_0u-Q_bu, 
 Q_0\w-Q_b\w\rangle_{\partial T}\Big|\\
&\leq  \Big(\sum_{T\in {\cal T}_h}h_T^{-1}\|    \kappa\nabla Q_0u\cdot \textbf{n}-Q_g( \kappa\nabla  u\cdot \textbf{n})\|^2_{\partial T}\Big)
^{\frac{1}{2}}\cdot\\
 &\qquad\Big(\sum_{T\in {\cal T}_h}h_T^{-1}\| \kappa\nabla Q_0\w\cdot \textbf{n}-Q_g( \kappa\nabla  \w\cdot \textbf{n})\|^2_{\partial T}\Big)^{\frac{1}{2}} \\
  &\qquad+ \Big(\sum_{T\in {\cal T}_h}h_T^{-3}\|  Q_0u-Q_bu\|^2_{\partial T}\Big)^{\frac{1}{2}}
\Big(\sum_{T\in {\cal T}_h}h_T^{-3}\|  Q_0\w-Q_b\w\|^2_{\partial T}\Big)^{\frac{1}{2}} \\
&\leq   Ch^{t_0-1}\|\w\|_{t_0+1}h^{k-1} \|u\|_{k+1} \\
&\leq   Ch^{k+t_0-2} \|u\|_{k+1} \|\w\|_{4}.\\
\end{split}
\end{equation*}

Finally, by substituting all the above estimates 
into (\ref{5.4}),   we obtain
$$
(1-Ch)\|e_0\|^2\leq C \big(h^{t_0-1} \3bar e_h\3bar +
h^{k+t_0-2}(\|u\|_{k+1}+h\delta_{k,2}\|u\|_4)\big)\|\w\|_4,
$$
which, together with the regularity estimate (\ref{5.2}) and (\ref{4}),  gives rise to the desired $L^2$ error estimate
(\ref{5.3}). This completes the proof of the theorem.
\end{proof}

\begin{theorem}\label{theoremeb}
 Let $u_h\in V_h$ be the solution of
the weak Galerkin algorithm (\ref{2.7}) with finite elements of
order $k\geq 2$. Let $t_0=\min\{k,3\}$. Assume that the exact
solution of (\ref{0.1}) is sufficiently regular such that   $u\in H^{\max\{k+1,4\}}(\Omega)$, and
the dual problem (\ref{5.1}) has the $H^4$ regularity. Define 
\begin{equation*}\label{ebnorm}
 \|u_b\|=\Big(\sum_{T\in {\cal
T}_h}h_T \|u_b\|^2_{\partial T} \Big)^{\frac{1}{2}}.
\end{equation*}
 There
exists a constant $C$ such that
\begin{equation*}\label{5.3u0}
\|Q_bu-u_b\|\leq
Ch^{k+t_0-2}\Big(\|u\|_{k+1}+h\delta_{k,2}\|u\|_{4}\Big).
\end{equation*}
In other words, we have a sub-optimal order of convergence for $k=2$
and optimal order of convergence for $k\geq 3$.
\end{theorem}
\begin{proof}
Letting 
$v=\{0,e_b,0\}$ in the error equation (\ref{4.1}), we obtain
\begin{equation}\label{ee1} 
\begin{split}
&(E_{w} e_h,E_{w}v)_h+2\mu (\kappa \nabla_w e_h,\nabla_w v)_h +\sum_{T\in {\cal T}_h}h_T^{-3} \langle e_0 -e_b, -e_b \rangle_{\partial T}\\
=&-\sum_{T\in {\cal T}_h}\langle  \kappa  \nabla (Eu- {\cal
Q}_h(Eu)  ) \cdot \textbf{n},
  -e_b \rangle_{\partial T} \\
 &+ \sum_{T\in {\cal T}_h} 2\mu  \langle  -e_b,   (\kappa\nabla u -    {\cal Q}_1 (\kappa  \nabla u ))\cdot \bn   \rangle_{\partial T}  
+\sum_{T\in {\cal T}_h} h_T^{-3} \langle Q_0u-Q_bu, -e_b \rangle_{\partial T}.\\
\end{split}
\end{equation}

 Letting $\varphi=E_we_h$ in  (\ref{2.4}) and $\boldsymbol{\psi}=\nabla_w e_h$ in  (\ref{2.4-2})  yields
 \begin{equation*}\label{2.4q}
  (E_{ w}v,E_w e_h)_T= 
- \langle e_b ,\kappa \nabla (E_w e_h)\cdot  \textbf{n}\rangle_{\partial T},
 \end{equation*}
 \begin{equation*}\label{2.4q-2}
  ( \nabla_{ w}v,\nabla_w e_h)_T=  
 \langle e_b ,\nabla_w e_h\cdot  \textbf{n}\rangle_{\partial T},
 \end{equation*}
which, combining with re-arranging the terms involved in (\ref{ee1}), give
\begin{equation}\label{righ}
\begin{split}
&\sum_{T\in {\cal T}_h} h_T^{-3}\|  e_b\|^2_{\partial T}\\
=& \sum_{T\in {\cal T}_h}\{\langle e_b ,\kappa \nabla (E_w e_h)\cdot  \textbf{n}\rangle_{\partial T}-2\mu \kappa  \langle e_b ,\nabla_w e_h\cdot  \textbf{n}\rangle_{\partial T}+ h_T^{-3}\langle   e_0,  
 e_b\rangle_{\partial T}\\&+\langle  \kappa  \nabla (Eu- {\cal
Q}_h(Eu)  ) \cdot \textbf{n},
   e_b \rangle_{\partial T} \\
 &+  2\mu  \langle  -e_b,   (\kappa\nabla u -    {\cal Q}_1 (\kappa  \nabla u ))\cdot \bn   \rangle_{\partial T}  +h_T^{-3} \langle Q_0u-Q_bu, -e_b \rangle_{\partial T} \}.
\end{split}
\end{equation}

Each term on the right-hand side of (\ref{righ}) can be bounded as follows. For the first term of the right-hand side of   (\ref{righ}), using Cauchy-Schwarz inequality, the trace inequality (\ref{x})
and inverse inequality, we obtain
\begin{equation*}\label{aa1}
 \begin{split}
  &\Big|\sum_{T\in {\cal T}_h}\langle e_b ,\kappa \nabla (E_w e_h)\cdot \textbf{n}\rangle_{\partial T}\Big|\\
\leq &\Big(\sum_{T\in {\cal T}_h} h_T\|e_b\|_{\partial T}^2\Big)^{\frac{1}{2}}\Big(\sum_{T\in {\cal T}_h} h_T^{-1}\|\kappa \nabla (E_w e_h)\cdot \textbf{n}\|_{\partial T}^2\Big)^{\frac{1}{2}}\\
\leq & C\Big(\sum_{T\in {\cal T}_h} h_T\|e_b\|_{\partial T}^2\Big)^{\frac{1}{2}}\Big(\sum_{T\in {\cal T}_h} h_T^{-4}\| E_w e_h \|_{ T}^2\Big)^{\frac{1}{2}}\\
\leq & C h ^{-2} \|e_b\|  \3bar  e_h \3bar.
 \end{split}
\end{equation*}
 For the second term of  the right-hand side  of  (\ref{righ}), using Cauchy-Schwarz inequality and the trace inequality (\ref{x}), we obtain
\begin{equation*}\label{aa2-2}
 \begin{split}
  &\Big|\sum_{T\in {\cal T}_h} 2\mu \kappa  \langle e_b ,\nabla_w e_h\cdot  \textbf{n}\rangle_{\partial T}\Big|\\
\leq &\Big(\sum_{T\in {\cal T}_h} h_T^{-1}\|\nabla_w e_h\cdot  \textbf{n}\|_{\partial T}^2\Big)^{\frac{1}{2}}\Big(\sum_{T\in {\cal T}_h} h_T\| e_b\|_{\partial T}^2\Big)^{\frac{1}{2}}\\
\leq & Ch ^{-1} \|e_b\|  \3bar  e_h \3bar.
 \end{split}
\end{equation*}
 For the  third  term of  the right-hand side of   (\ref{righ}), using Cauchy-Schwarz inequality and the trace inequality (\ref{x}), we obtain
\begin{equation*}\label{aa2}
 \begin{split}
  &\Big|\sum_{T\in {\cal T}_h} h_T^{-3}\langle   e_0,  
 e_b\rangle_{\partial T}\Big|\\
\leq &\Big(\sum_{T\in {\cal T}_h} h_T^{-7}\|e_0\|_{\partial T}^2\Big)^{\frac{1}{2}}\Big(\sum_{T\in {\cal T}_h} h_T\| e_b\|_{\partial T}^2\Big)^{\frac{1}{2}}\\
\leq & C h ^{-4} \|e_b\|  \|e_0\|.
 \end{split}
\end{equation*}
 For the fourth  term of  the right-hand side of   (\ref{righ}), using Cauchy-Schwarz inequality and the estimate (\ref{3.6}), we obtain
\begin{equation*}\label{aa3}
 \begin{split}
  &\Big|\sum_{T\in {\cal T}_h}\langle  \kappa \nabla (Eu-{\cal Q}_h Eu) \cdot \textbf{n},
  -e_b \rangle_{\partial T}\Big|\\
\leq & Ch ^{-2}\Big(\sum_{T\in {\cal T}_h} h_T^3 \|\kappa \nabla (Eu-{\cal Q}_h Eu) \cdot \textbf{n}\|_{\partial T}^2\Big)^{\frac{1}{2}}\Big(\sum_{T\in {\cal T}_h} h_T\| e_b\|_{\partial T}^2\Big)^{\frac{1}{2}}\\
\leq & C h ^{k-3} \|e_b\| (\|u\|_{k+1}+h\delta_{k,2}\|u\|_4).
 \end{split}
\end{equation*}
For the fifth  term of the right-hand side of     (\ref{righ}), using Cauchy-Schwarz inequality and the estimate (\ref{3.8-2}) , we obtain
\begin{equation*}\label{aa4-2}
 \begin{split}
  &\Big|\sum_{T\in {\cal T}_h}  2\mu  \langle  -e_b,  
 (\kappa\nabla u -    {\cal Q}_1 (\kappa  \nabla u ))\cdot \bn   \rangle_{\partial T}\Big|\\
\leq &  C h^{-2} \Big(\sum_{T\in {\cal T}_h}  h_T \|e_b\|_{\partial T}^2\Big)^{\frac{1}{2}}
\Big(\sum_{T\in {\cal T}_h} h_T^{3}\| (\kappa\nabla u -    {\cal Q}_1 (\kappa  \nabla u ))\cdot \bn  \|_{\partial T}^2\Big)^{\frac{1}{2}}\\
\leq & C h ^{k-1} \|e_b\| \|u\|_{k+1}.
 \end{split}
\end{equation*}
For the last term of  the right-hand side   of (\ref{righ}), using Cauchy-Schwarz inequality and the estimate (\ref{3.8}), we obtain
\begin{equation*}\label{aa4}
 \begin{split}
  &\Big|\sum_{T\in {\cal T}_h} h_T^{-3}\langle  Q_0u-Q_bu,  
 -e_b\rangle_{\partial T}\Big|\\
\leq &  Ch ^{-2} \Big(\sum_{T\in {\cal T}_h} h_T^{-3} \|Q_0u-Q_bu\|_{\partial T}^2\Big)^{\frac{1}{2}}\Big(\sum_{T\in {\cal T}_h} h_T\| e_b\|_{\partial T}^2\Big)^{\frac{1}{2}}\\
\leq & C h ^{k-3} \|e_b\| \|u\|_{k+1}.
 \end{split}
\end{equation*}

Substituting all the above estimates into  (\ref{righ}) and  the shape-regularity assumption for the finite element partition ${\cal T}_h$ give 
\begin{equation*}
 \begin{split}
  h^{-4}\|e_b\|^2\leq & C\sum_{T\in {\cal T}_h}h_T^{-3}\langle e_b,e_b\rangle_{\partial T}\\
\leq & C \Big(h^{-2}\3bar e_h\3bar+h^{-1} \3bar e_h\3bar +h^{-4}\|e_0\|+h^{k-3}
(\|u\|_{k+1}\\
&+h\delta_{k,2}\|u\|_4)+h^{k-1}\|u\|_{k+1}+h^{k-3}\|u\|_{k+1}\Big)\|e_b\|,
 \end{split}
\end{equation*}
which, combining with (\ref{5.3}) and  (\ref{4}), yields
\begin{equation*}
 \begin{split}
  \|e_b\|  \leq & C\Big(h^2\3bar e_h\3bar+ \|e_0\|+h^{k+1}(\|u\|_{k+1}+h\delta_{k,2}\|u\|_4)+h^{k+1}\|u\|_{k+1}\Big) \\
\leq  & C\Big(h^2 h^{k-1} (\|u\|_{k+1}+h\delta_{k,2}\|u\|_4 )+  h^{k+t_0-2} (\|u\|_{k+1}+h\delta_{k,2}\|u\|_{4} )\\
&+h^{k+1}(\|u\|_{k+1}+h\delta_{k,2}\|u\|_4)+h^{k+1}\|u\|_{k+1}\Big) \\
 \leq  & C h^{k+t_0-2} (\|u\|_{k+1}+h\delta_{k,2}\|u\|_{4} ).
 \end{split}
\end{equation*}
This completes the proof.
\end{proof}

\begin{theorem}\label{theoremeg}
 Let $u_h\in V_h$ be the solution of
the weak Galerkin algorithm (\ref{2.7}) with finite elements of
order $k\geq 2$. Let $t_0=\min\{k,3\}$. Assume that the exact
solution of (\ref{0.1}) is sufficiently regular such that $u\in H^{\max\{k+1,4\}}(\Omega)$, and
the dual problem (\ref{5.1}) has the $H^4$ regularity. Define 
\begin{equation*}\label{ebnormg}
 \|u_g\|=\Big(\sum_{T\in {\cal
T}_h}h_T \|u_g\|^2_{\partial T} \Big)^{\frac{1}{2}}.
\end{equation*}
 There 
exists a constant $C$ such that
\begin{equation*}\label{5.3u0g}
\|Q_g(\kappa \nabla u_0 \cdot \textbf{n})-u_g\|\leq
Ch^{k+t_0-3}\Big(\|u\|_{k+1}+h\delta_{k,2}\|u\|_{4}\Big).
\end{equation*}
In other words, we have a sub-optimal order of convergence for $k=2$
and optimal order of convergence for $k\geq 3$.
\end{theorem}
\begin{proof}
Letting 
$v=\{0,0,e_g\}$ in the error equation (\ref{4.1}), we obtain
\begin{equation}\label{ee2} 
\begin{split}
& (E_{w} e_h,E_{w}v)_h+2\mu (\kappa \nabla_w e_h,\nabla_w v)_h+ \sum_{T\in {\cal T}_h}
 h_T^{-1}\langle  \kappa \nabla
e_0 \cdot \textbf{n}- e_g,  -
e_g\rangle_{\partial T}  \\
=& \sum_{T\in {\cal T}_h}\langle  -e_g, Eu-{\cal Q}_hE u\rangle_{\partial
T} +\sum_{T\in {\cal T}_h} h_T^{-1}\langle  \kappa \nabla
Q_0u \cdot \textbf{n}- Q_g(\kappa \nabla u\cdot \bn),-
e_g\rangle_{\partial T} .
\end{split}
\end{equation}

 Letting $\varphi=E_we_h$ in  (\ref{2.4})   and $\boldsymbol{\psi}=\nabla_w e_h$ in  (\ref{2.4-2})  yields 
 \begin{equation*}\label{2.4qg}
   (E_{ w}v,E_we_h)_T= 
 \langle   e_{g },E_we_h \rangle_{\partial T},
 \end{equation*}
\begin{equation*}\label{2.4qg-2}
 ( \nabla_{ w}v,\nabla_w e_h)_T= 0,
 \end{equation*}
which, combining with re-arranging the terms  involved in (\ref{ee2}), yield
\begin{equation}\label{righg}
\begin{split}
&\sum_{T\in {\cal T}_h} h_T^{-1}\langle  e_g,  
 e_g\rangle_{\partial T}=
\sum_{T\in {\cal T}_h}
 -\langle   e_{g },E_we_h \rangle_{\partial T}+\sum_{T\in {\cal T}_h} h_T^{-1}\langle  \kappa\nabla e_0\cdot \textbf{n},  
 e_g\rangle_{\partial T}\\
&-\sum_{T\in {\cal T}_h}\langle e_g, Eu-{\cal Q}_h Eu\rangle_{\partial
T}+ \sum_{T\in {\cal T}_h}  h_T^{-1}\langle  \kappa \nabla
Q_0u \cdot \textbf{n}- Q_g(\kappa \nabla
 u_0 \cdot \textbf{n}), -
e_g\rangle_{\partial T}.
\end{split}
\end{equation}

Each term on the right-hand side of (\ref{righg}) can be bounded as follows. For the first term of   the right-hand side   of (\ref{righg}), 
using Cauchy-Schwarz inequality, the trace inequality (\ref{x}), we obtain
\begin{equation*}\label{aa1g}
 \begin{split}
  &\Big|\sum_{T\in {\cal T}_h}
 \langle   e_{g },E_we_h \rangle_{\partial T}\Big|\\
\leq &\Big(\sum_{T\in {\cal T}_h} h_T\|e_g\|_{\partial T}^2\Big)^{\frac{1}{2}}\Big(\sum_{T\in {\cal T}_h} h_T^{-1}\| E_w e_h \|_{\partial T}^2\Big)^{\frac{1}{2}}\\
\leq & C\Big(\sum_{T\in {\cal T}_h} h_T\|e_g\|_{\partial T}^2\Big)^{\frac{1}{2}}\Big(\sum_{T\in {\cal T}_h} h_T^{-2}\| E_w e_h \|_{  T}^2\Big)^{\frac{1}{2}}\\
\leq & C h ^{-1} \|e_g\|  \3bar  e_h \3bar.
 \end{split}
\end{equation*}
 For the second term of  the right-hand side   of (\ref{righg}), using Cauchy-Schwarz inequality and the trace inequality (\ref{x}) and inverse inequality, we obtain
\begin{equation*}\label{aa2g}
 \begin{split}
  &\Big|\sum_{T\in {\cal T}_h} h_T^{-1}\langle  \kappa\nabla e_0\cdot \textbf{n},  
 e_g\rangle_{\partial T}\Big|\\
\leq &\Big(\sum_{T\in {\cal T}_h} h_T^{-3}\| \kappa\nabla e_0\cdot \textbf{n}\|_{\partial T}^2\Big)^{\frac{1}{2}}\Big(\sum_{T\in {\cal T}_h} h_T\| e_g\|_{\partial T}^2\Big)^{\frac{1}{2}}\\
\leq &\Big(\sum_{T\in {\cal T}_h} h_T^{-6}\|e_0 \|_{  T}^2\Big)^{\frac{1}{2}}\Big(\sum_{T\in {\cal T}_h} h_T\| e_g\|_{\partial T}^2\Big)^{\frac{1}{2}}\\
\leq & C h ^{-3} \|e_g\|  \|e_0\|.
 \end{split}
\end{equation*}
For the third term of  the right-hand side   of  (\ref{righg}), using Cauchy-Schwarz inequality and the estimate (\ref{3.5}), we obtain
\begin{equation*}\label{aa3g}
 \begin{split}
  &\Big|\sum_{T\in {\cal T}_h}\langle e_g, Eu-{\cal Q}_h Eu\rangle_{\partial
T}\Big|\\
\leq & C h ^{-1}\Big(\sum_{T\in {\cal T}_h} h_T \| e_g\|_{\partial T}^2\Big)^{\frac{1}{2}}\Big(\sum_{T\in {\cal T}_h} h_T\|Eu
-{\cal Q}_h Eu\|_{\partial T}^2\Big)^{\frac{1}{2}}\\
\leq & C h ^{k-2} \|e_g\|  \|u\|_{k+1}.
 \end{split}
\end{equation*}
 For the last term of the right-hand side   of  (\ref{righg}), using Cauchy-Schwarz inequality and the estimate (\ref{3.7}), we obtain
\begin{equation*}\label{aa4g}
 \begin{split}
  &\Big|\sum_{T\in {\cal T}_h}  h_T^{-1}\langle  \kappa \nabla
Q_0u \cdot \textbf{n}- Q_g(\kappa \nabla
 u_0 \cdot \textbf{n}), -
e_g\rangle_{\partial T}\Big|\\
\leq &  Ch ^{-1} \Big(\sum_{T\in {\cal T}_h} h_T^{-1} \| \kappa \nabla
Q_0u \cdot \textbf{n}- Q_g(\kappa \nabla
 u_0 \cdot \textbf{n})\|_{\partial T}^2\Big)^{\frac{1}{2}}\Big(\sum_{T\in {\cal T}_h} h_T\| e_g\|_{\partial T}^2\Big)^{\frac{1}{2}}\\
\leq & C h ^{k-2} \|e_g\| \|u\|_{k+1}.
 \end{split}
\end{equation*}

Substituting all the above estimates into  (\ref{righg})   and the shape-regularity assumption for the finite element partition ${\cal T}_h$ give 
\begin{equation*}
 \begin{split}
  h^{-2}\|e_g\|^2\leq & C\sum_{T\in {\cal T}_h}h_T^{-1}\langle e_g,e_g\rangle_{\partial T}\\
\leq & C (h^{-1}\3bar e_h\3bar+h^{-3}\|e_0\|+h^{k-2} \|u\|_{k+1}  )\|e_g\|,
 \end{split}
\end{equation*}
which, together with (\ref{5.3}) and (\ref{4}), gives
\begin{equation*}
 \begin{split}
  \|e_g\|  \leq & C(h\3bar e_h\3bar+h^{-1}\|e_0\|+h^{k} \|u\|_{k+1}  ) \\
\leq & C\Big(h\cdot h^{k-1} (\|u\|_{k+1}+h\delta_{k,2}\|u\|_4 )+h^{-1}\cdot h^{k+t_0-2} (\|u\|_{k+1}+h\delta_{k,2}\|u\|_{4} )  +h^{k} \|u\|_{k+1}\Big) \\
 \leq  & C h^{k+t_0-3} (\|u\|_{k+1}+h\delta_{k,2}\|u\|_{4} ).
 \end{split}
\end{equation*}
This completes the proof.
\end{proof}

\section{Numerical Tests}

In this section, we present some numerical results for the WG finite
element method analyzed in the previous sections. The goal is to
demonstrate the efficiency and the convergence theory established
for the method. For the simplicity of
 implementation, the weak function $v=\{v_0,v_b,v_g\}$ can be  discretized by polynomials of degree of $k$,  $k-1$ and $k-1$, respectively.
We could obtain all the same error estimates as obtained in the previous sections and the analysis could be derived  without any difficulty\cite{zz2014}. Details are omitted here.

In our experiments, we implement the lowest order (i.e., $k=2$) scheme
for the weak Galerkin algorithm (\ref{2.7}). In other words, the
implementation makes use of the following finite element space
$$
\widetilde{V}_h=\{v=\{v_0,v_b,v_g\}, v_0\in P_2(T), v_b\in
P_1(e), v_g\in P_1(e), T\in {\cal T}_h, e\in
 \E_h \}.
$$

For any given $v=\{v_0,v_b,v_g\}\in \widetilde{V}_h$, the
discrete weak second order elliptic operator $E_{ w} v$ is
computed as a constant locally on each element $T$ by solving the
following equation
  \begin{equation*}
  \begin{split}
  (E_{ w}v,\varphi)_T=(v_0,E\varphi)_T- 
 \langle v_b ,\kappa \nabla \varphi\cdot  \textbf{n}\rangle_{\partial T}+
 \langle  v_{g },\varphi  \rangle_{\partial T}, 
 \end{split}
 \end{equation*}
for all $\varphi \in P_0(T)$. Since $\varphi \in P_0(T)$, the above
equation can be simplified as
  \begin{equation*}
  \begin{split}
(E_{ w}v,\varphi)_T= 
 \langle  v_{g },\varphi  \rangle_{\partial T}. 
 \end{split}
 \end{equation*}

 The WG finite element scheme
(\ref{2.7}) was implemented on  uniform
triangular partition, which was obtained by partitioning the
domain into $n\times n$ sub-squares  and then  dividing each square element into two
triangles by the diagonal line with a negative slope. The mesh size
is denoted by $h=1/n$.

  Table \ref{NE:TRI:Case1-1} shows the numerical results for the
exact solution  $u=x^2(1-x)^2y^2(1-y)^2$. The numerical experiment is conducted 
 on the unit square domain $\Omega=(0,1)^2$. This case has   homogeneous boundary
conditions for both Dirichlet and Neumann. We take the coefficient matrix   $\kappa=[1/(3 (1+0.01)),0;0,1/(3 (1+0.01))]$ and $\mu=0.01$
in the whole domain $\Omega$.
 The results indicate that the
convergence rate for the solution of the weak Galerkin algorithm (\ref{2.7}) is of order $O(h)$ in the discrete $H^2$ norm, and is of
order $O(h^2)$ in the standard $L^2$ norm.  The numerical results are in good consistency with theory
for the  $H^2$  and $L^2$ norm of the error. 
 Figure \ref{homo}
illustrates the WG numerical solution for the mesh size  $1/64$, which totals to 4096 elements.

\begin{table}[H]
\begin{center}
 \caption{Numerical error and convergence order for the exact solution $u=x^2(1-x)^2y^2(1-y)^2$.}\label{NE:TRI:Case1-1}
 \begin{tabular}{|c|c|c|c|c|}
\hline
$1/n$        & $\|u_0 -Q_0u\| $ & order in $L^2$  norm &  $\3bar u_h -Q_hu\3bar $  & order in $H^2$  norm   \\
\hline
  1         & 0.05458   &   & 0.09913 &     \\
\hline
2  &0.02163   & 1.33 &  0.06649  &   0.58  \\
\hline
4 &  0.006307   &  1.78  &  0.03643  & 0.87   \\
\hline
8 &0.001716 &  1.88 &  0.01904 &  0.94  \\
\hline
16  &4.582e-04   &   1.91 &  0.009830   &  0.95  \\
\hline
32  &   1.181e-04  & 1.96   &0.004988   &   0.98  \\
\hline
64  &  2.981e-05 &   1.99  &  0.002506&    0.99\\
\hline
\end{tabular}
\end{center}
\end{table}

 %\begin{wrapfigure}{r}{0.35\textwidth} % Inline image example
  \begin{figure}[ht]
  \begin{center}
  \includegraphics[width=0.35\textwidth]{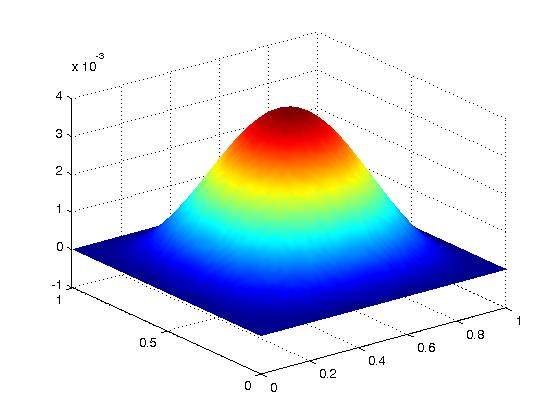}
  \caption{WG finite element solution for the exact solution $u=x^2(1-x)^2y^2(1-y)^2$ with mesh size 1/64.}\label{homo}
 \end{center}
  \end{figure}
 %\end{wrapfigure}

Table  \ref{NE:TRI:Case2-1} presents the numerical results when the exact solution is given by $u=\sin(\pi x)\sin(\pi y)$ 
on the unit square domain $\Omega=(0,1)^2$, which has  nonhomogeneous boundary
conditions.  The coefficient matrix $\kappa$ and the constant $\mu$ are taken to be the same values as the previous test.
  It shows that the
convergence rates for the solution of the weak Galerkin algorithm (\ref{2.7}) in
the $H^2$ and $L^2$ norms are of order $O(h)$ and $O(h^2)$,
respectively, which are in consistency with theory
for the $L^2$ and $H^2$ norms of the error. Figure \ref{nonhomo} gives the WG numerical solution for   the mesh size  $1/64$.

\begin{table}[H]
\begin{center}
\caption{Numerical error and convergence order for the exact solution
$u=\sin(\pi x)\sin (\pi y)$.}\label{NE:TRI:Case2-1}
\begin{tabular}{|c|c|c|c|c|}
\hline
$1/n$        & $\|u_0 -Q_0u\| $ & order in $L^2$  norm  &  $\3bar u_h -Q_hu\3bar$  & order in $H^2$  norm   \\
\hline
  1         &  5.587  &    &   10.20  &     \\
\hline
2  &  2.143   &  1.38  &  6.590 &  0.63   \\
\hline
4 &     0.6017  & 1.83    &   3.526 &  0.90\\
\hline
8&  0.1549 &  1.96  &   1.793 &  0.98  \\
\hline
16  &    0.03904  &   1.99  & 0.9005 &   0.99 \\
\hline
32 & 0.009783     &  2.00   & 0.4508    &  1.00  \\
\hline
64&   0.002447   & 2.00  &   0.2255   & 1.00  \\
\hline
\end{tabular}
\end{center}
\end{table}

   %\begin{wrapfigure}{r}{0.35\textwidth} % Inline image example
 \begin{figure}[ht]
\begin{center}
\includegraphics[width=0.35\textwidth]{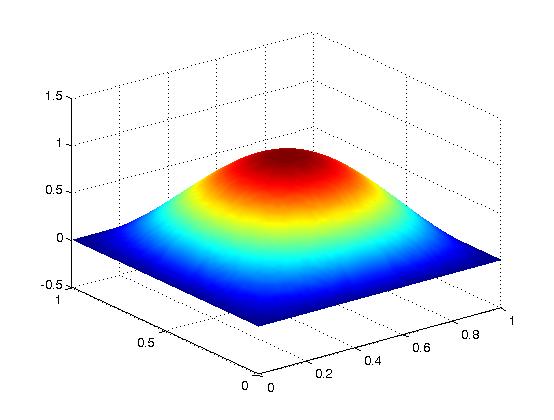}
\caption{WG finite element solution for the exact solution  $u=sin(\pi x)sin (\pi y)$ with mesh size 1/64.} \label{nonhomo}
\end{center}
 \end{figure}
% \end{wrapfigure}

In the rest of this section,  we shall conduct three different  types of  numerical tests arising from the FT model. Firstly, we take  $\kappa=[1,0;0,1]$ in 
the subdomain $\Omega_0=(1/4,3/8)^2$  and    $\kappa=[10^{-5},0;0,10^{-5}]$ in the rest of the whole domain $ \Omega=(0,1)^2$.
We take $\mu$ to be 0 in the whole domain  $ \Omega$.
The right-hand side $f$ is taken to be zero. For the  Dirichlet boundary condition  $u=\xi$, we take the boundary function $\xi$ 
as a piecewise continuous function which takes value 1 on the middle of each  boundary
segment and takes  value 0 on all the other edges for each of the four  boundary edges.   As to the Neumann boundary
condition $\kappa\nabla u\cdot \textbf{n}= \nu$, we take 
  $\nu=- \xi$ in the test. Figure
\ref{real1} illustrates the WG finite element solution for the mesh size  $1/64$.

%\begin{wrapfigure}{r}{0.35\textwidth} % Inline image example
 \begin{figure}[ht]
\begin{center}
\includegraphics[width=0.35\textwidth]{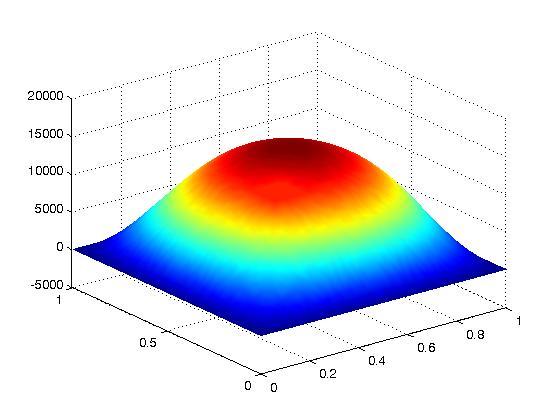}
\caption{WG finite element solution  with
discontinuous coefficients  with mesh size 1/64.} \label{real1}
\end{center}
 \end{figure}
%\end{wrapfigure}

Secondly, we
take the same $\kappa$,  $\mu$ and $f$ as in the previous test.  The Dirichlet boundary condition $u=\xi$
 is set as an approximate Direc-$\delta$ function on
each of the four  boundary edges. More precisely, this boundary data assumes
value $\frac{1}{|e|}$ on the middle edge of each boundary segment,
and takes value 0 on all the other edges. As to the Neumann boundary
condition $\kappa\nabla u\cdot \textbf{n}= \nu$, we take 
  $\nu=- \xi$ in the test. Figure
\ref{real2} illustrates the WG finite element solution for the mesh size  $1/64$.

%\begin{wrapfigure}{r}{0.35\textwidth} % Inline image example
 \begin{figure}[ht]
\begin{center}
\includegraphics[width=0.35\textwidth]{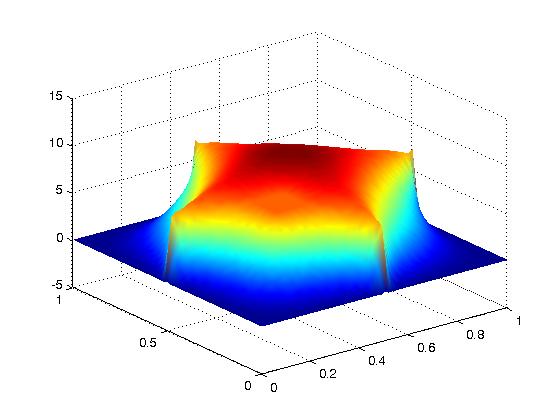}
\caption{WG finite element solution with
discontinuous   coefficients  with mesh size 1/64 with different Dirichlet boundary condition set as an approximate Direc-$\delta$ function.} \label{real2}
\end{center}
 \end{figure}
%\end{wrapfigure}

At last, we consider a real problem  arising from FT model which is implemented on the domain $\Omega=(0,50)^2$. There are two blocks in the domain $\Omega$: one 
is at (25,15) with radius  4; the other is at (35,20) with radius 3. 
The right hand side  data $f$ is the function modeling the light source, and we use Gaussian function  to model each point source $(x_0,y_0)$, with their centers 
locating  around the boundary of the domain.  
More precisely, we set $f=\sqrt{2 \pi \epsilon} e^{-\frac{(x-x_0)^2+(y-y_0)^2}{2\epsilon}}$ with $\epsilon=100/64$. The coefficient matrix $\kappa$ is taken to be   
$[1/(3 (1+0.01)),0;0,1/(3 (1+0.01))]$  and $\mu$ is taken to be 0.01 in the whole domain  $\Omega$ which arise from the data of the real
 problem in FT.  Figures
\ref{source1}-\ref{source3} illustrate  the WG finite element solution for   light sources with the mesh size  $1/64$, where the coordinates of    light sources are
(13.3065,    0.0730994) and
(49.8272,    13.5234), respectively.

In the future, we plan to conduct more numerical experiments   for the weak Galerkin
algorithm (\ref{2.7}), particularly for elements of order higher
than $k=2$ and for the real problems arising from the FT model. There is also a need of developing fast solution
techniques for the matrix problem arising from the WG finite element
scheme (\ref{2.7}). Numerical experiments on finite element
partitions with arbitrary polygonal element should be conducted for
a further assessment of the WG method.

 %\begin{wrapfigure}{r}{0.35\textwidth} % Inline image example
 \begin{figure}[ht]
\begin{center}
\includegraphics[width=0.35\textwidth]{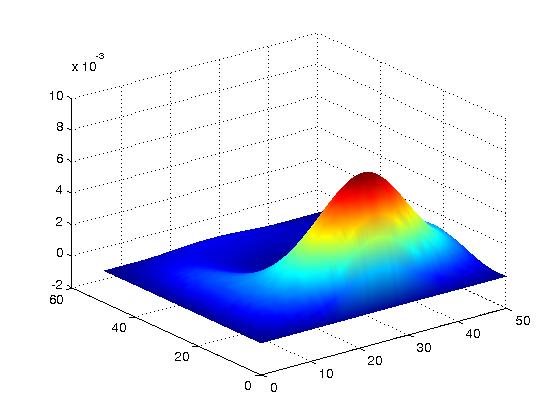}
\caption{WG finite element solution for a test case with
source point (13.3065,    0.0730994)  with mesh size 1/64.} \label{source1}
\end{center}
 \end{figure}
%\end{wrapfigure}

%\begin{wrapfigure}{r}{0.35\textwidth} % Inline image example
 \begin{figure}[ht]
\begin{center}
\includegraphics[width=0.35\textwidth]{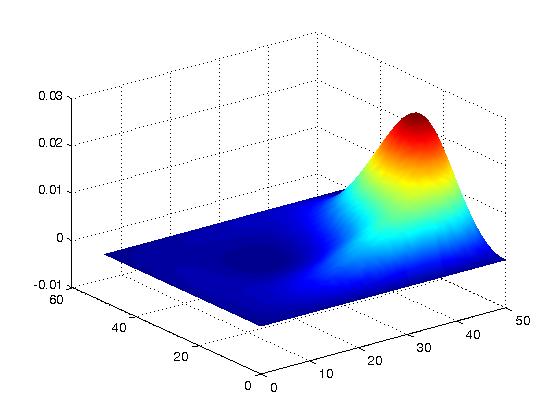}
\caption{WG finite element solution for a test case with
source point (49.8272,    13.5234)  with mesh size 1/64.} \label{source3}
\end{center}
 \end{figure}
%\end{wrapfigure}

\end{document}